\documentclass[11pt, table, nothms]{amsart}
\usepackage[a4paper,left=1in,right=1in,top=1in,bottom=1in]{geometry}
\usepackage{graphicx} 
\usepackage{amsfonts,amssymb,amsthm,amsmath}
\usepackage{mathtools}

\usepackage[ruled, linesnumbered]{algorithm2e}

\usepackage{standalone}
\usepackage[T1]{fontenc}
\usepackage{mathabx}
\usepackage[normalem]{ulem}

\usepackage[
hyperindex,
colorlinks,
breaklinks,
linkcolor={blue!50!black},
citecolor={blue!50!black},
urlcolor={blue!80!black}]{hyperref}

\usepackage[capitalise,nameinlink]{cleveref}
\crefname{step}{step}{step}
\crefname{property}{property}{property}
\crefname{condition}{condition}{condition}

\usepackage{tikz-cd}
\usetikzlibrary{decorations.markings}

\usepackage{quiver}
\usepackage{enumitem}

\usepackage{thmtools, thm-restate}

\usepackage{color}
\usepackage{xcolor}

\newtheorem{theorem}{Theorem}[section]
\newtheorem{proposition}[theorem]{Proposition}
\newtheorem{lemma}[theorem]{Lemma}
\newtheorem{corollary}[theorem]{Corollary}

\theoremstyle{definition}
\newtheorem{example}[theorem]{Example}
\newtheorem{definition}[theorem]{Definition}
\newtheorem{remark}[theorem]{Remark}

\newcommand{\qc}{,\quad}

\newcommand{\B}{\mathcal{B}}
\newcommand{\D}{D}

\newcommand{\rH}{H}
\newcommand{\Z}{\mathbb{Z}}
\newcommand{\C}{C}

\newcommand{\cK}{\mathcal{K}}
\newcommand{\cS}{\mathcal{S}}

\newcommand{\cH}{\mathcal{H}}
\newcommand{\sM}{\mathsf{M}}

\newcommand{\fA}{\mathfrak{A}}
\newcommand{\fB}{\mathfrak{B}}
\newcommand{\fD}{\mathfrak{D}}
\newcommand{\fH}{\mathfrak{H}}
\newcommand{\fK}{\mathfrak{K}}

\newcommand{\id}{\mathrm{id}}

\newcommand{\oC}{\widebar{C}}
\newcommand{\od}{\widebar{d}}

\newcommand{\bN}{\mathbb{N}}

\newcommand{\vp}{\widecheck{p}}

\newcommand{\ima}{\mathrm{im}}

\usepackage{bbm}
\usepackage{tikz}
\usetikzlibrary{fit}
\newcommand{\tikznode}[2]{\relax
\ifmmode%
  \tikz[remember picture,baseline=(#1.base),inner sep=0pt] \node (#1) {$#2$};
\else
  \tikz[remember picture,baseline=(#1.base),inner sep=0pt] \node (#1) {#2};%
\fi}
\tikzset{box around/.style={
    draw,
    inner sep=2pt,outer sep=0pt,
    node contents={},fit=#1
},      
}

\newcommand{\fibre}[2]{{#1}^{-1} \left(  #2 \right)}
\usepackage[
backend=biber,
style=alphabetic,
sorting=ynt
]{biblatex}
\usepackage{xurl} 

\crefname{equation}{eq.}{eqs.}

\begin{document}

\title{Computing Connection Matrices of Conley Complexes \\ via Algebraic Morse Theory}
\author{\'Alvaro Torras-Casas}
\email{atorras@us.es}
\address{Departamento de Matem\'atica Aplicada I, Universidad de Sevilla, España}

\author{Ka Man Yim}
\email{yimkm@cardiff.ac.uk}
\address{School of Mathematics, Cardiff University, UK}

\author{Ulrich Pennig}
\email{pennigu@cardiff.ac.uk}
\address{School of Mathematics, Cardiff University, UK}


\begin{abstract} 
    Given a poset-graded chain complex of vector spaces, a Conley complex is the minimal chain-homotopic reduction of the initial complex that respects the poset grading. A connection matrix is a matrix representing  the differential of the Conley complex.
    In this work, we give an algebraic derivation of the Conley complex and its connection matrix using homological perturbation theory and algebraic Morse theory.
    Under this framework, we use homology decompositions of relative chain complexes to determine the connection matrix, rather than Forman's acyclic partial matching in the usual discrete Morse theory setting. 
    This homology decomposition is obtained by means of the clearing optimisation, a commonly used technique in persistent homology.
    Finally, we show how this algebraic perspective yields an algorithm for computing the connection matrix via column reductions on the differential of the initial complex. 
\end{abstract}

\maketitle

\section{Introduction}

Morse theory is the study of homological data compression via gradient flows. In classical Morse theory, the singular homology of a compact manifold can be recovered from a \emph{Morse complex}, a chain complex generated by critical points of a Morse function, and whose boundary operator expresses relations between critical points by gradient flow. 
In discrete Morse theory~\cite{Forman1998MorseComplexes}, Forman transplants smooth Morse theory to the setting of finite CW complexes, where the role of the gradient {is enacted by} an \emph{acyclic partial matching} of cells. 
Given an acyclic partial matching, we can construct a smaller chain complex chain homotopic to the cellular chain complex, with chains generated by only unmatched `critical' cells.

Discrete Morse theory can be further abstracted algebraically as a chain homotopic reduction of \emph{poset-graded chain complexes}, whose links with Morse theory emerged from \emph{Conley theory}, an extension of Morse-theoretic constructs to non-gradient dynamical systems~\cite{Franzosa,Robbin1992LyapunovFunctor}. 
Given a poset-graded chain complex $(C,d)$ of vector spaces, its \emph{Conley complex} $(\oC, \od)~$\cite[Def.~4.23]{harker2021}  is the \emph{minimal} chain homotopic reduction that respects the poset grading (see \Cref{rmk:con_index}), satisfying the following conditions  given in \cref{def:connection_matrix}:
\begin{enumerate}
    \item There is a chain homotopy equivalence between $(\oC, \od)$ and $(C, d)$ that is compatible with the poset filtration; and
    \item Restricted within each poset grade, the differential $\od$ is trivial.
\end{enumerate}
A matrix {representation $\Delta$ of the differential} $\od$ is called a \emph{connection matrix}.  
The chain groups of the Conley complex, called Conley indices, consist of relative homology cycles of $C$ supported at each poset grade, generalising the Morse indices associated to critical points (cells) in smooth (discrete) Morse theory. We can consider a Conley complex to be the minimal representative of its equivalence class of complexes in the homotopy category of poset-graded chain complexes.

In this study, we focus on establishing an algebraic derivation of the Conley complex and an associated algorithm for computing the connection matrix {$\Delta$}. 
Our main  technique is based on the algebraic Morse theory of Sk\"oldberg  for complexes with fixed bases, or \emph{based complexes}~\cite{Skoldberg2006MorseViewpoint,Skoldberg2018}. Generalising Forman's acyclic matching of cells, Sk\"oldberg introduced the notion of `Morse matching' for based complexes, which reproduces discrete Morse theory on an algebraic level. Using homological perturbation theory introduced by~\cite{Brown1965TheTheorem, Gugenheim1972OnFibration}, we show how a Morse matching fashioned out of \emph{homology decompositions}~(\cref{eq:homology-decomposition}) of relative chain complexes enables us to reduce a poset-graded chain complex to its Conley complex. This derivation, which is our main algebraic contribution, is expressed in the theorem below (we recall a poset is well-founded if there are no infinite descending sequences of elements in $P$, e.g. when $P$ is finite).

\begin{restatable}{theorem}{mainthm} 
      Let $(C, d)$ be a $P$-graded chain complex of vector spaces where $P$ is well-founded.
      A homology decomposition of the relative chains $C^p$ of $(C,d)$ at each poset grade $\zeta^p_n: C^p_n \xrightarrow{\cong} H^p_n \oplus B^p_n \oplus B^p_{n-1}$ induces 
       a $P$-filtered contraction of $(C,d)$ to a Conley complex $(\oC, \od)$ 
    \begin{equation} \label{eq:main_contraction}
    \begin{tikzcd}[ampersand replacement=\&]
    (\oC, \od) \ar[r, "\bar{\beta}", shift right=0.3em, swap] \&
    (C,d) 
    \arrow[l, "\bar{\alpha}", shift right=0.3em, swap]
     \arrow["\bar{\gamma}", from=1-2, to=1-2, loop, in=345, out=15, distance=10mm]
    \end{tikzcd}.
    \end{equation}
where $\oC_n = \bigoplus_{p \in P} H_n^p$, and  $\od = \alpha \delta(1-\gamma\delta)^{-1}\beta$ for $\alpha, \beta, \gamma, \delta$ constructed from $d$ and $\zeta^p$ in \cref{eq:contraction_relative,eq:abc_proxy,eq:delta_graded}.
\label{thm:main}
\end{restatable}
Following the algebraic description of a connection matrix, we show how a homology decomposition of relative chains can be used to construct a four step algorithm for computing the connection matrix, given a matrix representation of $d$ (see \Cref{subsec:algorithm}). Since our algebraic description is based on a choice of splitting, we first use the clearing optimisation routine of~\cite{Chen2011,Bauer2014ClearChunks} to express the differential of the chain complex in a \emph{separating basis}, which effects the homology decomposition. 
After bringing the differential into the separating basis, and after (optionally) eliminating some columns and rows,
we can then perform column reductions on the matrix to produce the connection matrix. The final step then consists of eliminating the redundant columns and rows to isolate the {Conley} indices.
We prove the correctness of our algorithm in~\Cref{prop:connection-matrix-algo}. 

\subsection*{Related Work}

We recommend~\cite{Mischaikow1995ConleyTheory} for a general introduction to Conley index theory~\cite{Conley1978IsolatedIndex}, and~\cite{Mrozek2025ConnectionDynamics} for a contemporary account of Conley complexes in the context of combinatorial dynamics. The concept of a connection matrix in Conley theory was first introduced in~\cite{Franzosa}. This topic was further explored and applied in works such as~\cite{Franzosa1986IndexDecompositions,Reineck1990TheFlows,Robbin1992LyapunovFunctor,Franzosa1998AlgebraicTheory}. Recently~\cite{harker2021} established categorical equivalences between their version of a connection matrix and those of~\cite{Franzosa} and~\cite{Robbin1992LyapunovFunctor}, and showed that their definition of a connection matrix, which we follow, is unique up to a choice of $P$-filtered basis. In~\cite{Robbin1992LyapunovFunctor,Mrozek2025ConnectionDynamics}, they showed the existence Conley complexes by implicitly constructing it by induction.

We refer the reader to~\cite{Nicolaescu2011AnTheory} for comprehensive guide to smooth Morse theory and its applications and extensions in other areas of pure mathematics, such as~\cite{Witten1982SupersymmetryTheory,Floer1987MorseDiffeomorphisms}. A classic summary of the contributions of Morse theory to other areas of mathematics was written by~\cite{Bott1988MorseIndomitable}. For texts on discrete Morse theory, the reader can consult~\cite{Kozlov2021OrganizedTheory}, or~\cite{Knudson2015MorseDiscrete,Scoville2019DiscreteTheory} which describe both smooth and discrete aspects. Alongside Sk\"oldberg's algebraic generalisation of Morse theory, we also direct the reader to~\cite{Kozlov2005DiscreteComplexes} for a similar approach. We note that the algebraic Morse theory of~\cite{Skoldberg2006MorseViewpoint}
has also been applied to other computational topology problems as well, such as signal compression on simplicial complexes~\cite{Ebli2024MorseComplexes}, and the computation of cellular sheaf cohomology~\cite{Curry2016DiscreteCohomology}. 

Until recently, the focus of computing the connection matrix is in the setting where the initial chain complex is derived from a cellular complex, where Forman's discrete gradients (acyclic partial matchings)~\cite{Forman1998MorseComplexes} can be leveraged in the computation~\cite{Mischaikow2013MorseHomology,Harker2014DiscreteMaps,Harker2021MorseComputation,harker2021}. Algorithms that are purely algebraic, without reliance on any discrete gradients, have been recently proposed by~\cite{Dey2024ComputingReductions,Dey2025ComputingDecomposition}. The computation of an optimal acyclic partial matching that minimises the critical cells is $NP$-hard, equivalent to solving an integer programming problem~\cite{Joswig2004ComputingFunctions}. In \cite{harker2021,Harker2021MorseComputation}, the authors showed that efficient heuristics for computing non-optimal matchings, such as the \texttt{MorseReduce} algorithm~\cite{Mischaikow2013MorseHomology}, can be incorporated in a recursive routine to compute a connection matrix without directly solving the integer programming problem. 

\subsection*{Our contribution}

In this work, we show that choosing homology decompositions at each poset grade, one obtains the chain contraction from a poset-graded complex to its Conley complex by directly applying algebraic Morse theory from \cite{Skoldberg2018}.
Furthermore, we not only present an algebraic derivation of the Conley complex, but also translate the algebraic derivation into an algorithm for computing the connection matrix, following a line of recent work by~\cite{harker2021,Dey2024ComputingReductions,Mrozek2025ConnectionDynamics}. The development of the algorithm is enabled by the explicit description of the Conley complex in terms of algebraic formulae, a consequence of applying homological perturbation theory to homology decompositions. 
A benefit of our approach is that we obtain an explicit expression for the entries of the connection matrix for a given choice of separating bases, see~\Cref{cor:entry-connection}.

Our algorithm for computing the connection matrix is similar to that of the reduction algorithms of~\cite{Dey2024ComputingReductions,Dey2025ComputingDecomposition}; like theirs, our approach relies essentially on Gaussian elimination to compute the Conley complex in the general algebraic setting, and has complexity $O(N^3)$, where $N$ is the dimension (size of the basis) of the chain complex. 
Our approach differs in that we express the connection matrix explicitly in a relative homology basis, and the homology decomposition is explicitly computed (see~\Cref{alg:clear-reduce}). On the other hand,~\cite{Dey2024ComputingReductions,Dey2025ComputingDecomposition} implicitly infer the homology decomposition of relative chains in their routine, and the connection matrix obtained is $P$-filtered equivalent to the one we obtain in the relative homology basis.

\subsection*{Organisation}
In \Cref{sec:background}, we introduce the main concepts and theoretical background material for this article. In \Cref{ssec:chain-complex-field}, we set out the relevant concepts relating to chain complexes, namely contractions, homology decompositions, and the homological perturbation lemma (\Cref{lemma:perturbation}). In \Cref{ssec:pgradedcc}, we define poset-graded chain complexes, and the Conley complex. Last but not least, we give an account of Sk\"oldberg's algebraic Morse theory in~\Cref{ssec:AMT}, and introduce based complexes and Morse matchings. 

We prove our main algebraic result, \Cref{thm:main}, in \cref{sec:connectionmatrix}. \Cref{ssec:relhomperturb} gives the proof purely from the perspective of homological perturbation theory. In \Cref{ssec:conley_amt}, we show how we can arrive at the same derivation of the Conley complex by transforming the $P$-graded complex into a based complex. Readers more familiar with discrete Morse theory may find the derivation in~\Cref{ssec:conley_amt} more congenial as it generalises concepts in discrete Morse theory; for those  more focused on homological algebra, the proof in~\Cref{ssec:relhomperturb} is self-contained, and we do not rely on concepts introduced in~\Cref{ssec:AMT,ssec:conley_amt} for the proof or subsequent sections.

Our second main result is an algorithm to compute the connection matrix described in \cref{sec:connection-matrix-algorithm}. 
We start this section by showing that clearing optimisation~\cite{Chen2011, Bauer2014ClearChunks} leads to a homology decomposition of chain complexes over a field as well as to a separating basis associated to it; see~\cref{def:sep_basis}.
The algorithm then follows the steps outlined in \cref{subsec:algorithm} and a pseudocode instantiation is given in~\cref{alg:connection-matrix}.
We conclude this section with the proof of \cref{prop:connection-matrix-algo}, which shows that the result of the algorithm is indeed a connection matrix. 

A conclusion and a discussion on future research directions can be found in \cref{sec:conclusion}.

\section{Background} \label{sec:background}

In this section, we introduce the necessary background concepts for the article. We first review standard definitions of chain complex and contractions, and we detail relevant properties of $P$-graded chain complexes. In \Cref{ssec:pgradedcc}, we introduce the Conley complex. At the end, we briefly review algebraic Morse theory from the point of view of the perturbation lemma.

\subsection{Chain Complexes over a Field}\label{ssec:chain-complex-field}
We consider a \emph{chain complex} $(C, d)$ of vector spaces $C=\{C_n\}_{n \in \Z}$ over some field, together with \emph{differentials} $d$, which are linear maps $d_n\colon C_n \rightarrow C_{n-1}$ satisfying $d_{n-1}d_n=0$ for all $n \in \Z$. 
We always write $C$ to refer to a chain complex $(C, d)$, dropping the explicit mention of the differential. 
Given $r \in \Z$, we write $C[r]$ for the $r$-shifted chain complex given by $C[r]_n = C_{n+r}$ for all $n \in \Z$; the differential is also shifted similarly. 
Given two chain complexes, $(C, d^C)$ and $(D, d^D)$, a map of degree $r\in \Z$, $g\colon C\rightarrow D[r]$, is a collection of linear maps $g_n\colon C_n\rightarrow D_{n+r}$ for all $n \in \Z$. In this context, a map $f\colon C \rightarrow D[r]$ is a \emph{chain map} if the collection of linear maps $f_n\colon C_n \rightarrow D_{n+r}$ is such that $d^D_{n+r}f_n=f_{n-1}d^C_n$ for all $n \in \Z$. A chain complex $D$ is a \emph{subcomplex} of another chain complex $C$ if $D_n\subset C_n$ for all $n \in \Z$ and the differentials of $D$ are restrictions of the ones of $C$.

\subsubsection{Homology Decomposition}
Let $C$ be a chain complex.
The kernel and image of the differential $d$, denoted by $\ker(d)$ and $\ima(d)$ respectively, are subcomplexes of $C$.
Recall $\ker(d_n)$ and $\ima(d_{n+1})$ are called the $n$-cycles and $n$-boundaries of $C_n$, and the \emph{homology} group is given by the quotients 
$\rH_n(C)= \ker(d_n)\big/ \ima(d_{n+1})$. 
For each dimension $n \in \Z$, we can choose to decompose the chain groups as a direct sum
\begin{equation}\label{eq:splitting-homology}
\zeta: C_n \xrightarrow{\cong} H_n \oplus B_n \oplus K_{n}\ ,
\end{equation}
where {we use the notation  $B_n\coloneqq \ima(d_{n+1})$, and  $K_{n}\coloneqq B_{n-1}$. 
The decomposition is induced by considering the following two short exact sequences of vector spaces; the first
\[
\begin{tikzcd}
    0 \ar[r] &
    \ker(d_n) \ar[r, hookrightarrow ] &
    C_n \ar[r, "d_n", twoheadrightarrow] &
     \ima(d_n) \ar[r] &
    0
\end{tikzcd},
\]
allows us to choose a splitting $C_n \cong Z_n \oplus K_{n}$, where 
$Z_n 
\coloneqq \ker(d_n)$.
We have another splitting $Z_n\cong H_n\oplus B_n$ obtained from the exact sequence
$
\begin{tikzcd}
    0 \ar[r] &
    \ima(d_{n+1}) \ar[r, hookrightarrow] &
    \ker(d_n) \ar[r, twoheadrightarrow] &
    \rH_n(C) \ar[r] &
    0
\end{tikzcd}
$.

An arbitrary set of isomorphisms $\{\zeta_n \colon C_n\xrightarrow{\cong} H_n \oplus B_n \oplus K_n\}_{n \in \Z}$ need not be mutually compatible across chain groups of different dimensions. The set of isomorphisms $\{\zeta_n\}$ is said to be a \emph{homology decomposition (of $C$)} (c.f. ~\cite[Prop~3.4.5.]{Mrozek2025ConnectionDynamics}) if the following diagram commutes for all $n$:
    \begin{equation}
    \begin{tikzcd}[column sep=large]
	{C_n} & {H_n \oplus B_n \oplus K_n } & {K_n} \\
	{C_{n-1}} & {H_{n-1} \oplus B_{n-1} \oplus K_{n-1} } &  B_{n-1}
	\arrow["{\zeta_n^{-1}}"',"\cong", from=1-2, to=1-1]
	\arrow["{d_n}"', from=1-1, to=2-1]
	\arrow[hook', from=1-3, to=1-2]
	\arrow["{\zeta_{n-1}}","\cong"',  from=2-1, to=2-2]
	\arrow["=", from=1-3, to=2-3]
    \arrow[two heads,  from=2-2, to=2-3]
\end{tikzcd} \label{eq:homology-decomposition}
\end{equation}
The isomorphisms involved in splitting the short exact sequences do satisfy the consistency requirement. 
The existence of homology decompositions is a standard result which follows from $C$ being a chain complex over a field; see for example the introduction to chapter~1.4. from~\cite{Weibel1994AnAlgebra} or Proposition~3.4.5 from~\cite{Mrozek2025ConnectionDynamics}.
In \Cref{subsec:splittings-clearing}, we study how to compute such decompositions in practice.
\begin{remark}
    Notice that the splitting in \cref{eq:splitting-homology} is \emph{unnatural}.
    This means that, given a chain map $f\colon C \rightarrow D$, the morphism $f_n$ is not the direct sum of three chain maps between the subcomplexes $H, B, K$ of $C$ and $D$. 
\end{remark}

\begin{example}
    In this example, we illustrate how a set of isomorphisms  $\{\zeta_n \colon C_n\xrightarrow{\cong} H_n \oplus B_n \oplus K_n\}_{n \in \Z}$ do not necessarily lead to a homology decomposition.     
    Consider the simplicial complex $[0,1]$ giving the unit interval with $0$-simplices $\{0\}$ and $\{1\}$ and $1$-simplex $[0,1]$. Its nontrivial chain groups over $\Z_2$ are $C_1=\langle [0,1]\rangle$ and $C_0=\langle [0], [1]\rangle$, while its nontrivial subcomplexes of homology, boundaries and preboundaries are $B_0=K_1=\langle [0]+[1]\rangle$ and $H_0 = C_0/B_0$. 
    Then, consider a pair of isomorphisms $\zeta_1\colon C_1\rightarrow K_1$, given by the assignment $[0,1]\mapsto [0]+[1]$, and $\zeta_0 \colon C_0\rightarrow H_0 \oplus B_0$, given by the assignments $[0]\mapsto ([0]+B_0, 0)$ and $[1]\mapsto (0, [0]+[1])$. 
    In this case, one can check that $\zeta_0(d_1(\zeta_1^{-1}(K_1)))=([0]+B_0, [0]+[1])$, which is not contained in $B_0$. Hence, $\{\zeta_0, \zeta_1\}$ is not a homology decomposition. 
    On the other hand, changing $\zeta_0$ for $\zeta'_0 \colon C_0\rightarrow H_0 \oplus B_0$, given by the assignment $[0]\mapsto ([0]+B_0, 0)$ and $[0]+[1]\mapsto (0, [0]+[1])$, leads to a homology decomposition $\{\zeta_0', \zeta_1\}$.
\end{example}

\subsubsection{Chain Homotopy and Contractions}
A chain homotopy equivalence between two chain complexes $(C, d)$ and $(C',d')$ consists of chain maps $f,g$, and a collection of maps $h_n: C_n \to C_{n+1}$ and $k_n: C'_n \to C'_{n+1}$, 
\begin{equation}\label{eq:chain_homotopy}
    \begin{tikzcd}[ampersand replacement=\&]
	(C',d') \& (C,d)
	\arrow["k"', from=1-1, to=1-1, loop, in=195, out=165, distance=10mm]
	\arrow["g"', shift right, from=1-1, to=1-2]
	\arrow["f"', shift right, from=1-2, to=1-1]
	\arrow["h"', from=1-2, to=1-2, loop, in=15, out=345, distance=10mm]
    \end{tikzcd},
\end{equation}
such that the following are satisfied: 
\begin{equation}\label{eq:chain_homotopy_b}
fg = 1_{C'} + d'k  + kd',\quad \text{and} \quad 
gf = 1_C + dh + hd.
\end{equation}
Recall that a chain homotopy equivalence induces isomorphisms between the homology groups $C$ and~$C'$. A  \emph{contraction} $(f,g,h)$ is a chain homotopy equivalence  
\begin{equation} \label{eq:contraction}
    \begin{tikzcd}[ampersand replacement=\&]
	(C',d') \& (C,d)
	\arrow["g"', shift right, from=1-1, to=1-2]
	\arrow["f"', shift right, from=1-2, to=1-1]
	\arrow["h"', from=1-2, to=1-2, loop, in=15, out=345, distance=10mm]
\end{tikzcd},
\end{equation}
 where the linear map $k$ in \cref{eq:chain_homotopy_b} is trivial, and $f,g,h$ satisfy the identities
\begin{equation}\label{eq:contraction_b}
fg = 1_{C'},\   
gf = 1_C + dh + hd,\ 
fh = 0,\  
hg = 0,\mbox{ and } 
h^2=0.
\end{equation}

It follows from \cref{eq:contraction_b} that we also have $h + hdh=0$.
We recall a standard example how contractions arise from homology decompositions (see, for example,~\cite[Prop.~3.4.5.]{Mrozek2025ConnectionDynamics} and~\cite[\S 1.4]{Weibel1994AnAlgebra}).
\begin{example} \label{ex:basic_contraction}
A homology decomposition $\zeta_n: C_n \xrightarrow{\cong} H_n \oplus B_n \oplus K_{n}$ induces a contraction 
    \[
    \begin{tikzcd}[ampersand replacement = \&]
	(H(C), 0) \& (C,d)
	\arrow["\beta"', shift right, from=1-1, to=1-2]
	\arrow["\alpha"', shift right, from=1-2, to=1-1]
	\arrow["\gamma"', from=1-2, to=1-2, loop, in=15, out=345, distance=10mm]
\end{tikzcd},
    \]
    where $(a, \beta, \gamma)$ are defined for each dimension by
\[\begin{tikzcd}[ampersand replacement  = \&, column sep=small, row sep = small]
{\alpha_n:\ C_n} \& {H_n \oplus B_n \oplus K_n} \& {H_n} \\
	{\beta_n:\ H_n} \& {H_n \oplus B_n \oplus K_{n}} \& {C_n} \\
	{\gamma_n:\ C_n} \& {H_n \oplus B_n \oplus K_{n}} \& {B_n=:K_{n+1}} \& {H_{n+1} \oplus B_{n+1} \oplus K_{n+1}} \& {C_{n+1}} \& {C_{n+1}}
	\arrow["{\zeta_n}", "\cong"', from=1-1, to=1-2]
	\arrow[two heads, from=1-2, to=1-3]
	\arrow[hook, from=2-1, to=2-2]
	\arrow["{\zeta_n^{-1}}", "\cong"', from=2-2, to=2-3]
	\arrow["{\zeta_n}", "\cong"', from=3-1, to=3-2]
	\arrow[two heads, from=3-2, to=3-3]
	\arrow[hook, from=3-3, to=3-4]
	\arrow["{\zeta_{n+1}^{-1}}", "\cong"',  from=3-4, to=3-5]
    \arrow["-1",  from=3-5, to=3-6]
\end{tikzcd}\]
\end{example}
The example above shows how a homology decomposition reduces a chain complex to a minimal homotopy equivalent complex, composed of homology groups with trivial differentials. 
Our work considers the generalisation of this elementary result to poset-graded chain complexes, which we introduce below. 
Our main result, \Cref{thm:main}, describes how homology decompositions of \emph{relative} chains induce a contraction on the total complex that is compatible with the poset-grading. 
The smaller complex resulting from the contraction is  precisely the Conley complex, which we introduce in \Cref{def:connection_matrix}. We first introduce the main device with which we prove~\Cref{thm:main}, which is the homological perturbation lemma of Brown and Gugenheim~\cite{Brown1965TheTheorem, Gugenheim1972OnFibration}. 


\subsubsection{Homological Perturbation}\label{ssec:homological_perturbation}
Homological perturbation allows us to derive contractions of chain complexes by perturbing known contractions of other complexes. 
A \emph{perturbation} $\delta$ of a chain complex $(A, d)$ is a collection of degree minus one maps $\delta_n: A_n \to A_{n - 1}$ such that $(d+\delta)^2 = 0$. We call $(A, d + \delta)$ the perturbed complex. Given a $P$-graded chain complex $(C,d)$, our aim is to build a trivial contraction of a proxy complex $(A,d_A)$, and a perturbation $\delta$ such that $(C, d) = (A, d_A + \delta)$. We can then apply the following perturbation lemma to obtain a contraction of $(C,d)$.

For the lemma below, we recall a linear transformation $f: V \to W$ is locally nilpotent if for any $v \in V$, there is a finite $n$ such that $f^n(v) = 0$.

\begin{lemma}[{\cite{Brown1965TheTheorem, Gugenheim1972OnFibration}}]\label{lemma:perturbation}
Suppose we are given a contraction
\[
\begin{tikzcd}[ampersand replacement=\&]
    (C', d') \ar[r, "g", shift right=0.3em, swap] \&
    (C, d) 
    \arrow[l, "f", shift right=0.3em, swap]
    \arrow["h", from=1-2, to=1-2, loop, in=345, out=15, distance=10mm]
\end{tikzcd}.
\]
If a perturbation $\delta$ of $C$ is such that the linear map $h \delta: \bigoplus_n C_n \to \bigoplus_n C_n$ is locally nilpotent, then we can obtain a contraction of the perturbed complex $(C, d + \delta)$
\[
\begin{tikzcd}[ampersand replacement=\&]
    (C',d' + \delta')\ar[r, "g'", shift right=0.3em, swap] \&
    {(C, d + \delta)}
    \arrow[l, "f'", shift right=0.3em, swap]
    \arrow["h'", from=1-2, to=1-2, loop, in=345, out=15, distance=10mm]
\end{tikzcd}
\]
where, writing $\cS = \delta(1- h\delta)^{-1} = \sum_{i=0}^\infty \delta (h\delta)^i$, the induced maps are given by
\begin{align}
    \delta' & = f \cS g, &  
    f' &= f + f \cS h, & 
    g' &= g + h \cS g,  & h' &= h +  h\cS h.
\end{align}
\end{lemma}
\begin{remark}
    We note that, in the original presentation of the perturbation lemma~\cite{Brown1965TheTheorem,Gugenheim1972OnFibration}, the chain complexes are required to admit a lower bounded exhaustive filtration; the contraction $(f,g,h)$ need to be compatible with the filtration; and the perturbation $\delta$ must be filtration lowering. These assumptions exist to ensure the sum in $\cS$ is well-defined. Here we subsume this by requiring $h\delta$ 
    to be locally nilpotent. 
\end{remark}

\subsection{Poset-graded Chain Complexes} \label{ssec:pgradedcc}

In this section, we follow the framework of~\cite{Robbin1992LyapunovFunctor,harker2021, Mrozek2025ConnectionDynamics}. Let $(P, \leq)$ be a poset. 
Henceforth, we impose the following finiteness assumption on the poset $P$.
We say that a partial order on a set $P$ is \emph{well-founded} if there is no strictly descending infinite sequence in $P$. 
That is, for any descending sequence $p_1 \geq p_2 \geq \cdots $, we have that $p_k = p_{k+1}$ for sufficiently large $k$. 

A vector space $V$ is $P$-graded if it admits a decomposition into a direct sum $V  = \bigoplus_{p \in P} V^p$ indexed by poset elements. 
For $p \in P$, we let $\imath^p: V^p \hookrightarrow V$ denote the natural inclusion, and $\jmath^p: V \twoheadrightarrow V^p$ denote the projection $(v_q)_{q \in P} \mapsto v_p$. 
In particular, we note that $\id_V = \sum_p \imath^p \circ \jmath^p$. 

For a linear map $f: V \to W$ between $P$-graded vector spaces $V = \bigoplus_{p \in P} V^p$ and $W = \bigoplus_{p \in P} W^p$, we denote by $f^{pq}$ the composition
\begin{equation*}
    \begin{tikzcd}
f^{pq}: V^q \ar[r, "\imath^q", hookrightarrow] &
V \ar[r, "f"] &
W \ar[r, "\jmath^p", twoheadrightarrow] &
W^p
\end{tikzcd}.
\end{equation*}
\begin{sloppypar}
We say $f$ is \emph{$P$-filtered} if $f^{pq} \neq 0 \implies p \leq q$. Composition of $P$-filtered linear maps are also $P$-filtered~\cite[Prop~ 4.3]{harker2021}: if $f: V \to W$ and $g: W \to U$ are $P$-filtered, then 
\begin{align} \label{eq:graded_compose}
    gf  = \sum_{p \leq  q}(gf)^{pq} \quad \text{and} \quad  (gf)^{pq}  = \sum_{p \leq r \leq q}g^{pr} f^{rq}.
\end{align} 
In particular, $(gf)^{pp}  = g^{pp} f^{pp}$. 
A chain complex $(C,d)$ is \emph{$P$-graded} if the chain group at each dimension $C_n = \bigoplus_{p \in P} \C^p_n$ is $P$-graded, and the differentials $d_n: C_{n} \to C_{n-1}$ are $P$-filtered. For any $p \in P$, we let $F^p_n = \bigoplus_{r \leq p} \C^{r}_n$, and $F^{\vp}_n = \bigoplus_{r \lneq p} \C^{r}_n$ denote the sublevel set and strict sublevel set of the $P$-graded chain complex at $p$. Since $d_n$ are $P$-filtered, and $F^p_n$ and $F^{\vp}_n$ are downward closed with respect to $\leq$, the restrictions of the total complex to $F^p$ and $F^{\vp}$ are also $P$-graded chain complexes. 
\end{sloppypar}

We say a chain map $f: C \to D$ between $P$-graded chain complexes is $P$-filtered if for each dimension $f_n: C_n \to D_n$ are $P$-filtered linear maps. We say an equivalence of $P$-graded chain complexes as in~\cref{eq:chain_homotopy} is $P$-filtered if linear maps $h_n,k_n$ and chain maps $f,g$ are $P$-filtered. In particular, we say a contraction $(f,g,h)$ in the form of~\cref{eq:contraction} is $P$-filtered if $f,g,h$ are all $P$-filtered. 

\begin{remark} \label{rmk:lattice}
    If $P$ is finite, then Birkhoff's representation theorem~(\cite{Birkhoff1933-jg}, or~\cite[Theorem 6.2,10.4]{Roman2008-fh}) relates the poset with the finite distributive lattice $\mathsf{O}(P)$ of downsets of $P$. Recall a subset $A \subset P$ is a downset if it is downward closed: for any $q \in A$, if $p \leq q$, then $p \in A$. For each $p \in P$, we can send it to the corresponding downset $\downarrow p := \{q \in P \ : \ q \leq p\}$. Birkhoff's representation theorem states that $\{\downarrow p\}_{p \in P}$ are precisely the elements of the lattice $\mathsf{O}(P)$ which have a unique predecessor in the lattice partial order --- the \emph{join-irreducible elements} $\mathsf{J}(\mathsf{O}(P))$ --- and this bijection is a poset isomorphism $P \cong\mathsf{J}(\mathsf{O}(P))$. We note that for $\downarrow p$, the unique predecessor in $\mathsf{O}(P)$ is $\downarrow p \setminus \{p\}$.
    Considering the lattice of downset $\mathsf{O}(P)$ allows us to view the $P$-graded chain complex as a lattice filtered complex, a perspective articulated in~\cite[\S 7.2]{harker2021} and~\cite[\S2]{spendlove2025graded}. An $\mathsf{O}(P)$-filtered vector space $V$ is one equipped with a lattice homomorphism $\mathrm{flt}: \mathsf{O}(P) \to \mathsf{Sub}(V)$, where $\mathsf{Sub}(V)$ is the lattice of vector subspaces with join $\vee = +$ and meet $\wedge = \cap$.  
    A $P$-graded chain complex $(C,d)$ is an instance of an $\mathsf{O}(P)$-filtered chain complex, in the following sense. The chain group in each dimension $C_n$ is an $\mathsf{O}(P)$-filtered vector space $\mathrm{flt}_n: \mathsf{O}(P) \to \mathsf{Sub}(C_n)$, where
    \begin{equation*}
        \mathrm{flt}_n(A) = \bigoplus_{p \in A} C^p_n.
    \end{equation*}
    Each differential $d_n: C_n \to C_{n-1}$ is $\mathsf{O}(P)$-filtered, meaning it satisfies $d_n(\mathrm{flt}_n(A)) \subseteq \mathrm{flt}_{n-1}(A)$ for all $A \in \mathsf{O}(P)$. Given this framework, we can then regard the sublevel set chain complexes as lattice filtered subcomplexes of $(C,d)$:
    \begin{equation*}
        \bigoplus_n F^p_n = \mathrm{flt}(\downarrow p)\qc \bigoplus_n F^{\vp}_n = \mathrm{flt}(\downarrow p \setminus \{p\}).
    \end{equation*}
\end{remark}

\subsubsection{Relative Chain Complexes}
We define, for $p,q \in P$, the \emph{twisted differentials} to be the linear maps 
\begin{equation}
\begin{tikzcd}
    t_n^{pq}: \C^q_n \ar[r, "\imath^q", hookrightarrow] & 
    C_n \ar[r, "d_n"] & 
    C_{n-1} \ar[r, "\jmath^p", twoheadrightarrow] & 
    \C^p_{n-1}. 
\end{tikzcd}
\end{equation}
In particular, since $d_n$ is $P$-filtered, it follows that $t_n^{pq}\neq 0$ implies $p \leq q$. 
We call $(\C, t)$ a \emph{multicomplex}\footnote{This is the same notion of a multicomplex in~\cite{Wall1961ResolutionsGroups,Gugenheim1974OnProducts,Livernet2020OnMulticomplex} if $P  = \Z$.} indexed by~$P$. To abbreviate notation we let $t^p := t^{pp}$. Because $d^2 = 0$, \cref{eq:graded_compose} implies the twisted differentials of a multicomplex satisfy the following identity for fixed $p \leq q$:
\begin{align} \label{eq:anticommute_steroids}
    \sum_{p \leq r \leq q}t^{pr}_{n-1} t^{rq}_{n}  
    &= (d^2)^{pq}= 0.
\end{align}
In particular,  $t^{p}_{n-1} t^{p}_n = 0$. This implies $(\C^p, t^p)$ are chain complexes in their own right. Because we have a short exact sequence of chain complexes,
\[\begin{tikzcd}
	0 & {F^{\vp}} & {F^p} & {\C^p} & 0 	\arrow[from=1-1, to=1-2]
	\arrow[hook, from=1-2, to=1-3]
	\arrow[two heads, from=1-3, to=1-4]
	\arrow[from=1-4, to=1-5]
\end{tikzcd},\]
the relative chain complex $F^p / F^{\vp}$ is isomorphic to $(\C^p, t^p)$. We hence refer to the chain complex $(\C^p, t^p)$ as the \emph{relative chain complex at $p$}.

\begin{remark}
    While  $t^{pq}_n: \C^q_n \to \C^p_{n-1}$ are linear maps between the chain groups, the relations between the twisted differentials in~\cref{eq:anticommute_steroids} imply the collection of linear maps $t^{pq}_n$ do not constitute a chain map from $(\C^q, t^{q})$ to $(\C^p, t^p)$, as they violate the commutativity condition required. 
    More concretely, generally one has that $t^{pq}_{n-1} \circ t^q_n\neq t^p_{n-1} \circ t^{pq}_n$.
    For instance, in \Cref{ex:relhom} twisted differentials \emph{anti-commute}. 
\end{remark}

\begin{remark} \label{rmk:relative_chain_homotopy} We also note that a $P$-filtered chain homotopy equivalence $C \simeq D$ descends to a chain homotopy equivalence of relative chain complexes $\C^p \simeq \D^p$ for all $p \in P$. The chain homotopy can be obtained by restricting the $P$-filtered maps $f,g,h,k$ in the chain homotopy described in~\cref{eq:chain_homotopy} to the relative chain complexes at each poset level $p$:
\[    \begin{tikzcd}
	\D^p & \C^p
	\arrow["k^p"', from=1-1, to=1-1, loop, in=195, out=165, distance=10mm]
	\arrow["g^p"', shift right, from=1-1, to=1-2]
	\arrow["f^p"', shift right, from=1-2, to=1-1]
	\arrow["h^p"', from=1-2, to=1-2, loop, in=15, out=345, distance=10mm]
\end{tikzcd}.\]
Note that the maps above satisfy the conditions for homotopy equivalence in~\cref{eq:chain_homotopy_b}; for example, since all the maps in sight are $P$-filtered, applying~\cref{eq:graded_compose} to the conditions in~\cref{eq:chain_homotopy_b} one obtains
\begin{align*}
  g^p f^p &= (gf)^p = (1_C + dh + hd)^p  =1_{\C^p}+ t^p h^p + h^p t^p.
\end{align*}
\end{remark}
\begin{example}[Relative homology] \label{ex:relhom} Let $P = \{0 \leq 1\}$. Consider a $P$-graded chain complex $(C,d)$ where $C_n \cong  \C_n^0 \oplus \C^1_n$. We have twisted differentials induced by $d$:
    \[\begin{tikzcd}
	{\C_n^0} & {\C_n^1} \\
	{\C^0_{n-1}} & {\C^1_{n-1}}
	\arrow["{t^{0}_n}", from=1-1, to=2-1]
	\arrow["{t^{01}_n}", from=1-2, to=2-1]
	\arrow["{t^{1}_n}", from=1-2, to=2-2]
\end{tikzcd},\]
which satisfy $t^{01}_{n-1} t^{1}_n +  t^{0}_{n-1} t^{01}_{n}= 0$. 
The sublevel sets are $F^0 = \C^0$ and $F^1 = C$. The chain complex $(\C^1, t^1)$ is isomorphic to the relative chain complex $F^1/F^0$.  The short exact sequence $F_n^{0} \to F_n^{1} \to F^1_n/F^0_n$ gives rise to the long exact sequence
\[
\begin{tikzcd}
   \dots \ar[r] & 
   \rH_n(\C^0) \ar[r] & 
   \rH_n(C) \ar[r] & 
   \rH_n(\C^1) \ar[r, "d_n"
   ] & 
   \rH_{n-1}(\C^0) \ar[r] & 
   \dots \ .
\end{tikzcd}
\]
We leave it as a standard diagram chasing exercise to the reader that the differential 
$
d_n\colon \rH_n(\C^1)  \to \rH_{n -1}(\C^0)
$ 
is induced by the twisted differential $t^{01}_n$. 
\end{example}

\subsubsection{The Conley Complex}
We say the differential $d$ of a $P$-graded chain complex is \emph{strict} if $t^p_n = 0$ for all $p \in P$ and $n \in \Z$. Note that $d$ is strict if and only if $\C^p_n \cong H_n(\C^p)$. Given a $P$-graded chain complex $(C,d)$, the main subject of our study is the chain homotopic reduction of $(C,d)$ to a strict chain complex, which is the \emph{Conley complex}. 

\begin{definition} \label{def:connection_matrix} Given a $P$-graded chain complex $(C,d)$, we say another $P$-graded chain complex $(\oC, \od)$ is a \emph{Conley complex} of $(C,d)$, if:
\begin{enumerate}[label = \emph{(M\arabic*)}, ref = (M\arabic*)]
    \item \label{M1} There is a $P$-filtered chain homotopy equivalence $(C,d) \simeq (\oC, \od)$; and
    \item  \label{M2} $\od$ is strict.
\end{enumerate}
We refer to the chain groups $\oC^p_n$ of the Conley complex as the \emph{Conley indices}, and we say that a matrix $\Delta$ representing  $\od$ is
a \emph{connection matrix} of $(C,d)$.
\end{definition}
If a Conley complex exists for a $P$-graded chain complex, then it is unique up to $P$-filtered chain isomorphisms. 
This is due to a remarkable fact~\cite[Proposition 4.27]{harker2021} about strict complexes: two strict $P$-graded chain complexes are $P$-filtered chain isomorphic, if and only if they are $P$-filtered chain homotopy equivalent. 
The existence of Conley complexes for finite $P$-graded chain complexes was first shown in~\cite[Theorem 8.1, Cor. 8.2]{Robbin1992LyapunovFunctor}; we also refer the reader to~\cite[Lemma 5.3.1, Theorem 5.3.2]{Mrozek2025ConnectionDynamics} for a more recent account of the proof.

\begin{remark} \label{rmk:con_index}
    We can regard the Conley complex $(\oC, \od)$ of $(C,d)$ as the unique $P$-graded homotopy equivalent complex of $(C,d)$ which is minimal within every poset grade. The Conley complex replaces the relative chain complex $\C^p$ at each poset grading with its minimal equivalent complex (in the sense presented in \Cref{ex:basic_contraction}). The relative chain complexes $(\oC^p, \od^p)$ at each grading of the Conley complex has trivial differentials $\od^p = 0$ (by definition \labelcref{M2}), and relative chain groups 
     \[\oC^p_n \cong \rH_n(\C^p). \] See~\cite[\S 8]{Robbin1992LyapunovFunctor}. This follows from the defining properties of the Conley complex: the strictness condition~\labelcref{M2} implies $\oC^p_n \cong \rH_n(\oC^p)$, the $P$-filtered equivalence~\labelcref{M2} implies via~\Cref{rmk:relative_chain_homotopy} the equivalence of relative chains $\oC^p \simeq \C^p$ at each poset level. Passing to homology, we thus obtain $\oC^p_n \cong \rH_n(\C^p)$. 
\end{remark}

Before we move on to discussing algebraic Morse theory, we first give a few example applications of $P$-graded chain complexes in mathematics. A key example of a $P$-graded chain complex is one derived from (multi-parameter) filtrations of simplicial complexes, a common occurrence in applied and computational topology.  
Let $X$ be a finite simplicial complex, and let $(C,d)$ be a chain complex associated to $X$. 
Viewing $X$ as a poset where the simplices are ordered by face relations, a filter function of $X$ is a monotone map $f: X \to P$ to another poset $P$. Given such a filter function $f$, we can express $(C,d)$ as a $P$-graded chain complex, where $\C^p_n$ are vector spaces whose basis are the $n$-simplices in $\fibre{f}{p}$. 
The chain complexes $F^p$ and $F^{\vp}$ are  then the chain complexes of the subcomplexes $\{f \leq p\}$ and $\{f \lneq p\}$ respectively. 

\begin{example} \label{ex:filtration}
Let $X = \Delta^2$ be the simplicial complex shown in \cref{fig:multipar}. The numbers on vertices, edges and the face indicate the values of a filter function $f \colon X \to P$, where $P$ is the linear poset $\{0 \leq  1 \leq 2 \leq 3\}$. 
The associated chain complex is $P$-graded.
The basis of $\C_1^2$, for example, is given by the two edges (i.e.\ $1$-simplices) marked by $2$ in \cref{fig:multipar}; these are also denoted as $uv$ and $vw$ using the labelling on the right hand side of~\cref{fig:multipar}. 
\begin{figure}[ht]
    \centering
\begin{tikzpicture}[scale=1.4]
    \draw[fill=blue!20] (0,0) -- (1,1) -- (-.5,1.5);
    \node[circle,fill=black,inner sep=0pt,minimum size=3pt,label=below:{$0$}] (a) at (0,0) {};
    \node[circle,fill=black,inner sep=0pt,minimum size=3pt,label=right:{$1$}] (b) at (1,1) {};
    \node[circle,fill=black,inner sep=0pt,minimum size=3pt,label=above:{$2$}] (c) at (-.5,1.5) {};
    \draw (a) -- node[below,xshift=4pt,yshift=2pt]{$1$} (b);
    \draw (b) -- node[above]{$2$} (c);
    \draw (c) -- node[left]{$2$} (a);    
    \node[inner sep=0pt,minimum size=0pt,label=below:{$3$}] (c) at (0.2,1) {};
    \begin{scope}[xshift=3cm]
        \draw[fill=blue!20] (0,0) -- (1,1) -- (-.5,1.5);
    \node[circle,fill=black,inner sep=0pt,minimum size=3pt,label=below:{$u$}] (a) at (0,0) {};
    \node[circle,fill=black,inner sep=0pt,minimum size=3pt,label=right:{$w$}] (b) at (1,1) {};
    \node[circle,fill=black,inner sep=0pt,minimum size=3pt,label=above:{$v$}] (c) at (-.5,1.5) {};
    \draw (a) -- (b);
    \draw (b) -- (c);
    \draw (c) -- (a);   
    \end{scope}
\end{tikzpicture}
    \caption{Left: The simplicial complex $X = \Delta^2$ together with an example of a filtration function $f \colon X \to P$ for the linear poset $P = \{0,1,2,3\}$. Right: the vertices of $\Delta^2$ labelled by $v,u,w$.}
    \label{fig:multipar}
\end{figure}
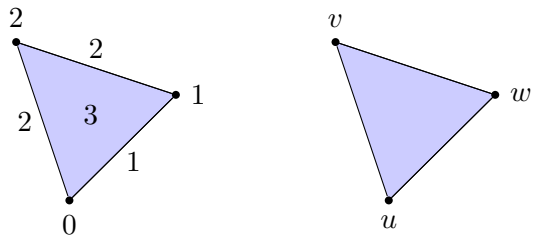
\end{example}

The concept of a $P$-graded chain complex {comes up} in Conley theory for dynamical systems~\cite{Conley1978IsolatedIndex,Mischaikow1995ConleyTheory}, from which we derive {terms} such as the Conley complex and connection matrix. 
Given a dynamical system, and a specified finite lattice of attractors, one can obtain a \emph{Morse decomposition} of the underlying space, a partition indexed over the poset $P$ of join-irreducible elements of the lattice~\cite{Robbin1992LyapunovFunctor,harker2021}. 
The partial order in the poset encodes attractor-repeller relationships between the attractors in the lattice they index. 
The partition $P$ then induces a splitting of the total singular chain complex of the domain as a $P$-graded chain complex. 
In that context, the connection matrix describes algebraic relations between the Conley indices of the isolated invariant sets associated to the Morse decomposition~\cite{Robbin1992LyapunovFunctor,Mischaikow1995ConleyTheory}.

Concrete computational examples of such $P$-graded chain complexes from Morse decompositions can be found in combinatorial dynamics~\cite{Batko2020LinkingDecompositions,Dey2022PersistenceSystems,Lipinski2023Conley-Morse-FormanSpaces,Dey2024ComputingReductions}. A \emph{multivector field} on a simplicial complex is a partition of the simplicial complex, where each partition element (multivector) is a convex subset of the face poset of the complex~\cite[Def 2.2]{Dey2024ComputingReductions}. This notion of a combinatorial vector field generalises Forman's discrete vector field~\cite{Forman1998MorseComplexes}, defined as a partial matching of cells. The associated multivector field generates a multi-valued function on simplices, and the combinatorial dynamical system is given by iterations of this multi-valued function. The Morse decomposition of the simplicial complex can then be obtained as the condensation of a directed graph associated to the dynamical system. We give a brief illustrative example below but refer the reader to~\cite[\S 3.4]{Dey2019PersistentDynamics} for technical details. 

\begin{example} \label{ex:multivector}
Consider the simplicial complex $X$ shown on the left hand side of \cref{fig:multivector} with vertex set $\{A,B,C,D,E\}$. The set 
\[
    \{\{B,AB\}, \{A,AE\}, \{E,ED\}, \{D,CD\}, \{C,BC\}, \{AD\}, \{AC, ABC\}, \{ADE\}\}
\]
is a multivector field on $X$. The partition into Morse sets and the corresponding poset $P$ are shown on the right of \cref{fig:multivector}. The dynamics on edges are indicated by the arrows. 
The chain complex of $X$ is $P$-graded. 
The chain group $C_1(X)$, for example, is $7$-dimensional with one basis element corresponding to the red edge, one to the blue edge and the remaining five corresponding to the trapezoid. 
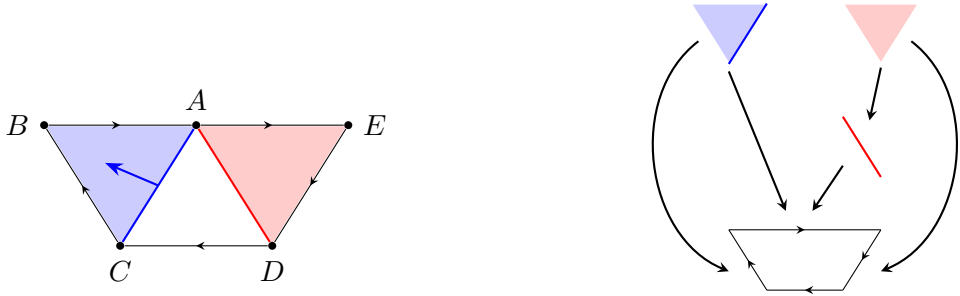
\begin{figure}[th]
\centering
\begin{align*}
\begin{tikzpicture}[decoration={markings,mark=at position 0.5 with {\arrow{stealth}}}] 
\draw[thick,draw=white, fill=blue!20] (2,1) -- (0,1) -- (1,-.6) -- (2,1) -- cycle;
\draw[thick,draw=white, fill=red!20] (2,1) -- (4,1) -- (3,-.6) -- (2,1) -- cycle;
\node[circle,fill=black,inner sep=0pt,minimum size=3pt,label=above:{$A$}] (A) at (2,1) {};
\node[circle,fill=black,inner sep=0pt,minimum size=3pt,label=left:{$B$}] (B) at (0,1) {};
\node[circle,fill=black,inner sep=0pt,minimum size=3pt,label=below:{$C$}] (C) at (1,-.6) {};
\node[circle,fill=black,inner sep=0pt,minimum size=3pt,label=below:{$D$}] (D) at (3,-.6) {};
\node[circle,fill=black,inner sep=0pt,minimum size=3pt,label=right:{$E$}] (E) at (4,1) {};
\draw[postaction={decorate}] (B) -- (A);
\draw[postaction={decorate}] (A) -- (E);
\draw[postaction={decorate}] (E) -- (D);
\draw[postaction={decorate}] (D) -- (C);
\draw[postaction={decorate}] (C) -- (B);
\draw[thick,draw=blue] (A) -- (C);
\draw[thick,draw=blue,-{Stealth[blue]} ] (1.5,.2) -- (0.8,.5);
\draw[thick,draw=red] (A) -- (D);
\end{tikzpicture}
& \hspace{3cm}
\begin{tikzpicture}[scale=.5,decoration={markings,mark=at position 0.5 with {\arrow{stealth}}}]
\draw[thick,draw=white, fill=blue!20] (2,1) -- (0,1) -- (1,-.6) -- (2,1) -- cycle;
\draw[thick,draw=white, fill=red!20, xshift=2cm] (2,1) -- (4,1) -- (3,-.6) -- (2,1) -- cycle;
\draw[thick,draw=blue] (2,1) -- (1,-.6);
\draw[thick,draw=red,yshift=-3cm,xshift=2cm] (2,1) -- (3,-.6);
\draw[yshift=-6cm,xshift=1cm,postaction={decorate}] (0,1) -- (4,1); 
\draw[yshift=-6cm,xshift=1cm,postaction={decorate}] (4,1) -- (3,-.6);
\draw[yshift=-6cm,xshift=1cm,postaction={decorate}] (3,-.6) -- (1,-.6);
\draw[yshift=-6cm,xshift=1cm,postaction={decorate}] (1,-.6) -- (0,1);
\draw[-stealth,thick] (1,-.8) -- (2.5,-4.5);
\draw[-stealth,thick] (5,-.7) -- (4.7,-2.1);
\draw[-stealth,thick] (.2,0) to[bend right=60] (1,-6.1);
\draw[-stealth,thick] (5.8,0) to[bend left=60] (5,-6.1);
\draw[-stealth,thick] (4,-3.3) -- (3.2,-4.5);
\end{tikzpicture}
\end{align*}
\caption{Left: the simplicial complex with the multivector field from Example~2.4. Right: the poset associated to the partition into Morse sets.}
\label{fig:multivector}
\end{figure}

\end{example}

\subsection{Algebraic Morse Theory for Based Complexes}\label{ssec:AMT}

We now introduce Sk\"oldberg's algebraic Morse theory framework~\cite{Skoldberg2006MorseViewpoint}. The main idea is to generalise acyclic partial matchings on cell complexes in discrete Morse theory~\cite{Forman1998MorseComplexes} to a purely algebraic setting. We first define the based complex, which generalises the notion of cells, and define Morse matchings, which is the algebraic anlogue of acyclic partial matchings. We then give a summary of Sk\"oldberg's derivation of algebraic Morse theory from homological perturbation theory (\cite{Skoldberg2018}), which concerns how Morse matchings can be used to derive contractions of chain complexes. 

\subsubsection{Based Complexes}
\begin{definition}
   A chain complex of vector spaces $(C, d)$ is a \emph{based complex} if for each vector space $C_n$, there is a direct sum decomposition $C_n \cong  \bigoplus_{a \in I_n} C_n^{(a)}$ as vector spaces, where $I = (I_n)_{n \in \Z}$ is a mutually disjoint collection of indexing sets.  
\end{definition}
Given a based complex $(C,d)$ for $a \in I_n$, there is an injection $\imath_n^{(a)} \colon C_n^{(a)} \to C_n$ and a projection $\jmath_n^{(a)} \colon C_n \to C_n^{(a)}$ such that $\jmath_n^{(a)} \imath_n^{(a)} = \id_n^{(a)}$, where $\id_n^{(a)}$ denotes the identity map on $C_n^{(a)}$. For a chain map of degree $r\in \Z$, $f\colon C \to C[r]$, and for $a \in I_{n+r}$ and $b \in I_{n}$, we denote by $f^{(a b)}$ the following composition
\[
\begin{tikzcd}
C_n^{(b)} \ar[r, "\imath_n^{(b)}", hookrightarrow] &
C_n \ar[r, "f"] &
C_{n+r} \ar[r, "\jmath_{n+r}^{(a)}", twoheadrightarrow] &
C_{n+r}^{(a)} 
\end{tikzcd}
\]
In particular, given $a \in I_{n-1}$ and $b \in I_n$, we consider the restriction of the differential $d^{(ab)}: C^{(b)}_{n} \to C^{(a)}_{n-1}$. 

Next, we give a concrete example of a based complex due to a homology decomposition of chain groups. 

\begin{example}\label{ex:homology_decomposition_based}
    Let $C$ be a chain complex equipped with a homology decomposition $\zeta_n: C_n \xrightarrow{\cong} H_n \oplus B_n \oplus K_n$. We can express $C$ as a based complex by writing the chain groups as 
    \[C_n \cong  \bigoplus_{V \in \{H, B,K\}} C^{(V)}_n \] 
    where $C^{(V)}_n := V_n$ for $V \in \{H, B,K\}$. Given the maps $\zeta_n$, the projections and inclusions are given by
    \begin{align*}
        \jmath_n^{(V)} &: C_n \xrightarrow[\cong]{\zeta_n} H_n \oplus B_n \oplus K_n \twoheadrightarrow V_n  \\ \imath_n^{(V)}&: V_n \hookrightarrow H_n \oplus B_n \oplus K_n \xrightarrow[\cong]{(\zeta_n)^{-1}} C_n .
    \end{align*}
For $V,W \in \{H, B, K\}$, the differentials restricted to each subgroup are
\begin{equation} \label{eq:based_restricted_diff}
    d_n^{(WV)} : V_n \xhookrightarrow{\imath^{(V)}_n} C_n \xrightarrow{d_n} C_{n-1} \overset{\jmath^{(W)}_n} {\twoheadrightarrow} W_n.
\end{equation}
As $\zeta$ is a homology decomposition, 
and since $(\zeta_n)^{-1}(H_n \oplus B_n) = \ker d_n$, the only non-zero map is where $V = K$ (i.e. when $d$ acts on the pre-boundaries). 
Furthermore,  the characterisation of homology decompositions in \cref{eq:homology-decomposition} implies  $d_n^{(BK)}: K_n \to B_{n-1}$ is the identity (by definition) $K_n := B_{n-1}$ on the right hand side of the commutative diagram, and the only non-zero block for $V,W \in \{H,B,K\}$. 
\end{example}

\subsubsection{Morse Matchings}
We now introduce the concept of Morse matchings as given in~\cite{Skoldberg2006MorseViewpoint}. Given a based complex $(C, d)$, we denote by $G(C)$ a directed graph with vertices indexed over the subspace indices $\bigsqcup_n I_n$, together with directed edges $b \rightharpoonup a$ whenever $d^{(ab)}\neq 0$, for all $n \in \Z$ with $b \in I_n$ and $a \in I_{n-1}$. 
A \emph{partial matching $\sM$ on a digraph} $G$ is a subset of directed edges $\sM$ such that no two edges in $\sM$ share a common vertex. 
We denote $G^\sM$ to be the digraph obtained by reversing the orientation of the directed edges in $\sM$. We use $a \rightharpoondown b$ to denote an edge in $G^\sM$ which corresponds to $b \rightharpoonup a$ in $\sM$.

\begin{example}\label{ex:homology_decomposition_based_2} 
    In the case where the based complex is given by a homology decomposition (\Cref{ex:homology_decomposition_based}), the digraph $G(C)$ has vertices indexed over $I_n = \{H_n, B_n, K_n\}$ for $n \in \Z$, and directed edges $K_n \rightharpoonup B_{n-1}$. We can construct a partial matching $\sM$ by including all such directed edges whenever $B_n \neq 0$. The digraph $G^\sM$ is then the directed graph with directed edges $B_{n-1} \rightharpoondown K_n$.
\end{example}
\begin{definition}\label{def:Morse-matching}
    Let $(C, d)$ be a based complex.
    A partial matching $\sM$ on $G(C)$ is a \emph{Morse matching} if:
    \begin{enumerate}
        \item $d^{(ab)}: C^{(b)} \to C^{(a)}$
        is an isomorphism for all 
        $(b \rightharpoonup a)
        \in \sM$; and
        \item We have a well-founded partial order $\preceq_n$ on $I_n$ satisfying the following: for $a \neq b$, if there is a directed path $a\rightharpoonup c \rightharpoondown b$ or $a\rightharpoondown  c  \rightharpoonup b$ in $G(C)^\sM$, then we have $a \succneq_n b$.
    \end{enumerate}
\end{definition}
If $\sM$ is a Morse matching, then we say $\sM^0$ is the \emph{critical set} of the matching, which consists of indexing elements in $\bigcup_{n \in \Z} I_n$ which are neither sources nor targets of any edge in~$\sM$.

\begin{example} \label{ex:homology_decomposition_based_3}
    Consider the case where the based complex is given by a homology decomposition (\Cref{ex:homology_decomposition_based}), and the partial matching $\sM$ given in \Cref{ex:homology_decomposition_based_2} where all the directed edges $K_n \rightharpoonup B_{n-1}$ are included. In this case, $\sM$ is a Morse matching. We first note that all non-trivial paths in $G(C)^\sM$ are of length $1$ from $B_{n-1} \rightharpoondown K_n$, and thus there are no relations between elements of $I_n$. The second condition of a Morse matching is thus trivially satisfied for $\sM$. Since $d_n^{(BK)}$ are isomorphisms, the first condition of the Morse matching is also satisfied for $\sM$.  Because all $K_n,B_n$'s are respectively sources and edges of the Morse matching, the critical set $\sM^0$ is given by the homology groups $\{H_n\}$. 
\end{example}
We note that~\Cref{def:Morse-matching} does not impose conditions on $G(C)$. 
However, when $G(C)$ is finite, there is a simple characterisation for Morse matchings.
\begin{lemma}[{\cite[Lemma 1]
{Skoldberg2006MorseViewpoint}}]\label{lemma:morse-cycles-matching}
    Let $(C, d)$ be a based complex such that the digraph $G(C)$ is finite, and let $\sM$ be a partial matching on $G(C)$ such that $d^{(ab)}$ is an isomorphism for all 
    $(b \rightharpoonup a)\in \sM$. 
    Then, $\sM$ is a \emph{Morse matching} if and only if $G(C)^\sM$ has no directed cycles.
\end{lemma}

\subsubsection{Contractions via Morse Matchings}
We now recall the main result of~\cite{Skoldberg2006MorseViewpoint,Skoldberg2018}. By applying the perturbation lemma (\Cref{lemma:perturbation}), Sk\"oldberg derives a contraction of a based complex $C$ given a Morse matching $\sM$. Let us denote by
\[
    \widetilde{C}_n = \bigoplus_{a \in \sM^0_n} C^{(a)}_n
\] 
the critical subspaces; we let $f_n : \widetilde{C}_n \leftrightarrows C_n: g_n$ be the natural inclusion and projection maps between the subspace $\widetilde{C}_n$ and $C_n$.  We also define $h_n: C_n \to C_{n+1}$ and $\tau_n: C_n \to C_{n - 1}$:
\begin{align}
    h_n &= -\sum_{(b\rightharpoonup a) \in \sM} \imath_{n+1}^{(b)} \circ (d_{n+1}^{(ab)})^{-1}  
    \circ \jmath_n^{(a)} & \tau_n &= \sum_{(b\rightharpoonup a) \notin \sM}\imath_{n-1}^{(a)} \circ d_n^{(ab)} 
    \circ \jmath_n^{(b)} . \label{eq:based_elementary_contr_maps}
\end{align}
Since $\sM$ is a Morse matching, for $b\rightharpoonup a$ in $\sM$, $d^{(ab)}$ is an isomorphism so $h$ is well-defined. By perturbing the contraction 
\begin{equation} \label{eq:basecontraction_morsematching}
    \begin{tikzcd}[ampersand replacement=\&]
    (\widetilde{C}, 0) \ar[r, "g", shift right=0.3em, swap] \&
    (C,d-\tau) 
    \arrow[l, "f", shift right=0.3em, swap]
     \arrow["h", from=1-2, to=1-2, loop, in=345, out=15, distance=10mm]
    \end{tikzcd}
\end{equation}
with $\tau$ on the right hand side, we can obtain the contraction~\cref{eq:amt_contraction} in \Cref{thm:skoldberg} via the perturbation lemma (\Cref{lemma:perturbation}). 

\begin{theorem}[{\cite[Lemma 2]{Skoldberg2018}}]\label{thm:skoldberg} Let $(C,d)$ be a based complex where a Morse matching $\sM$ exists. For $h,\tau$ as defined in~\cref{eq:based_elementary_contr_maps}, if $h\tau: \bigoplus C_n \to \bigoplus C_n$ is locally nilpotent, then we have a contraction 
\begin{equation} \label{eq:amt_contraction}
    \begin{tikzcd}[ampersand replacement=\&]
    (\widetilde{C}, \widetilde{d}) \ar[r, "\widetilde{g}", shift right=0.3em, swap] \&
    (C,d) 
    \arrow[l, "\widetilde{f}", shift right=0.3em, swap]
     \arrow["\widetilde{h}", from=1-2, to=1-2, loop, in=345, out=15, distance=10mm]
    \end{tikzcd},
\end{equation}
where $\widetilde{C}_n = \bigoplus_{a \in \sM^0_n} C^{(a)}_n$, and the induced maps are given by
\begin{align}
    \widetilde{d} & = f \cS g, &  
    \widetilde{f} &= f + f \cS h, & 
    \widetilde{g} &= g + h \cS g,  & \widetilde{h} &= h +  h\cS h.
\end{align}
for $f,g,h,\tau$ as defined in \cref{eq:based_elementary_contr_maps} and $\cS=\tau(1-h\tau)^{-1}$. 
\end{theorem}

\begin{example} \label{ex:homology_decomposition_based_4} Given a chain complex with a homology decomposition $\zeta$, we recover the contraction of the chain complex in \Cref{ex:basic_contraction} if we `apply' \Cref{thm:skoldberg} to the case where we consider the chain complex as a based complex via $\zeta$ (\Cref{ex:homology_decomposition_based}), and impose the Morse matching in \Cref{ex:homology_decomposition_based_2,ex:homology_decomposition_based_3} such that the critical sets are $\widetilde{C}_n = H_n(C)$. 
In this case, we verify that $\tau = 0$; thus $\widetilde{d} = 0$,  $h\tau$ is trivially nilpotent, and the contraction $(\widetilde{f},\widetilde{g},\widetilde{h})$ is exactly $(\alpha, \beta, \gamma)$ from \Cref{ex:basic_contraction}. This is because  $\widetilde{f} = f=\alpha$ and $\widetilde{g} = g =\beta$ are respectively the projection and inclusion maps between $H_n(C)$ and $C$. To see that $\widetilde{h}_n = \gamma_n$, we first note that  $\widetilde{h}_n = h_n = -\imath_{n+1}^{(K)} \circ (d_{n+1}^{(BK)})^{-1} \circ \jmath_n^{(B)}$, then show that $h_n = \gamma_n$ by expressing $\imath_{n+1}^{(K)},\jmath_n^{(B)}$ in terms of $\zeta$:
\begin{equation}
    h_n: C_n \xrightarrow{\zeta_n} H_n \oplus B_n \oplus K_n \twoheadrightarrow B_n \xrightarrow[=]{(d_{n+1}^{(BK)})^{-1}} K_{n+1}  \hookrightarrow H_{n+1} \oplus B_{n+1} \oplus K_{n+1} \xrightarrow{\zeta_{n+1}^{-1}} C_{n+1}\xrightarrow{-1} C_{n+1},
\end{equation}
where we recall that $d_{n+1}^{(BK)}$ recovers the equality $K_{n+1} := B_n$ from \Cref{ex:homology_decomposition_based}.
\end{example}

In \Cref{ssec:conley_amt}, we see how a $P$-graded chain complex equipped with homology decompositions of relative homology groups can be cast as a based complex, and thus allow us to apply \Cref{thm:skoldberg} to obtain a contraction. Compared to the case with trivial grading as described in \Cref{ex:homology_decomposition_based_4}, the perturbation $\tau$ and  differential $\widetilde{d}$  of the contracted complex is no longer trivial. 

\section{Algebraic Description of the Conley Complex}
In this section we prove our main result and provide a derivation of the Conley Complex of a $P$-graded chain complex $(C,d)$. In~\Cref{ssec:relhomperturb}, we achieve this by employing the homological perturbation lemma introduced in~\Cref{ssec:homological_perturbation}, which extends the contraction of relative homology groups $C^p$ induced by homology decompositions to a $P$-filtered contraction of the whole complex. In~\Cref{ssec:conley_amt}, we show that we can equivalently conceptualise the problem of obtaining a Conley Complex as that of finding a Morse matching on $(C,d)$ as a based complex, under the framework of \Cref{ssec:AMT}. We show that homology decompositions of $C^p$ allow us to express $(C,d)$ as a based complex and obtain a Morse matching. Under this choice, we arrive at the same $P$-filtered contraction of $(C,d)$ as the one produced with the first approach in~\Cref{ssec:relhomperturb}.

\label{sec:connectionmatrix}
\subsection{The Conley Complex via Contractions of Relative Homology Groups} \label{ssec:relhomperturb}
We first outline a construction that forms the heart of the proof of \Cref{thm:main}. Given a $P$-graded chain complex $(C,d)$, consider the contraction of relative chain complexes $(C^p, t^p)$ at each poset grade $p\in P$ given by elementary contractions induced by a homology decomposition (\Cref{ex:basic_contraction}):
\begin{equation} \label{eq:contraction_relative}
\begin{tikzcd}[ampersand replacement=\&]
    (H(C^p), 0) \ar[r, "\beta^p", shift right=0.3em, swap] \&
    (C^p, t^p)
    \arrow[l, "\alpha^p", shift right=0.3em, swap]
    \arrow["\gamma^p", from=1-2, to=1-2, loop, in=345, out=15, distance=10mm]
\end{tikzcd}
\end{equation}
The contractions above allow us to define a new contraction by taking direct sums over poset grades $p \in P$. On the right hand side of the contraction, the direct sum of $(C^p, t^p)$ over $p \in P$ yields a $P$-graded chain complex $(C,\widehat{d})$; the direct sum yields the same chain groups $C_n = \bigoplus_{p \in P} C^p_n$ of the initial complex $(C,d)$ by definition, and the differential is given by 
\begin{equation}
    \widehat{d} = \sum_{p \in P} \imath^p \circ t^p \circ \jmath^p 
\end{equation}
On the left hand side of the contraction, the direct sum yields a $P$-graded chain complex $(\oC, 0)$ with trivial differentials, where $\oC^p_n = H_n(C^p)$ and $\oC_n = \bigoplus_{p \in P} \oC^p_n$. We let $\bar{\imath}_n^p: \oC^p_n \hookrightarrow \oC_n$ and $\bar{\jmath}_n^p:  \oC_n \twoheadrightarrow \oC^p_n$ denote the natural inclusion and projection onto factors. Note that the chain groups of $\oC$ are isomorphic to those of the Conley complex $\oC$ of $C$, as $\oC^p_n \cong H_n(C^p)$ for all $p \in P$ (\Cref{rmk:con_index}). By taking direct sums over $\alpha^p, \beta^p, \gamma^p$, we obtain a contraction 
\begin{equation} 
\label{eq:Pgraded_elementary_contraction}
    \begin{tikzcd}[ampersand replacement=\&]
    (\oC, 0) \ar[r, "\beta", shift right=0.3em, swap] \&
    (C, \widehat{d})
    \arrow[l, "\alpha", shift right=0.3em, swap]
    \arrow["\gamma", from=1-2, to=1-2, loop, in=345, out=15, distance=10mm]
\end{tikzcd},
\end{equation}
where the maps are given by 
\begin{align}
    \alpha &= \sum_{p \in P} \bar{\imath}^p \circ  \alpha^p \circ \jmath^p, & \beta &= \sum_{p \in P} \imath^p \circ \beta^p \circ \bar{\jmath}^p, & \gamma &= \sum_{p \in P}\imath^p \circ \gamma^p \circ \jmath^p. \label{eq:abc_proxy}
\end{align}
We note that this is a contraction that only takes the differential within each level set into account.

The main idea of the proof is to perturb $(C, \widehat{d})$ by 
\begin{equation} \label{eq:delta_graded}
    \delta \coloneqq d - \widehat{d}=\sum_{p \lneq q} \imath^p \circ t^{pq} \circ \jmath^q.
\end{equation} 
Since $\widehat{d} =d-\delta$, 
we recover our original complex $(C,d)$ by perturbing $(C,\widehat{d})$ by $\delta$. In order to apply the perturbation lemma (\Cref{lemma:perturbation}) to derive a contraction of $(C,d)$, we first verify the local nilpotency condition. 

\begin{lemma}\label{lem:pgraded-locallynilp}
    For $\gamma$ and $\delta$ as defined in \cref{eq:abc_proxy,eq:delta_graded} respectively, $\gamma\delta$ is locally nilpotent for $P$ a well-founded poset. 
\end{lemma}
\begin{proof}
    To start, notice that $\gamma \delta =\sum_{p < q} \imath^p \circ (\gamma^pt^{pq})\circ \jmath^q$. 
    Thus, inductively, one can show that
    \begin{equation}~\label{eq:locally-nilpotent}
        (\gamma \delta)^n = \underbrace{\gamma \delta \circ \dots \circ \gamma\delta}_{n \text{ times}} = \sum_{p_n < \cdots < p_1 < p_0} \imath^{p_n} \circ (\gamma^{p_n} t^{p_n p_{n-1}}) \cdots (\gamma^{p_2}t^{p_2 p_1})(\gamma^{p_1}t^{p_1 p_0}) \circ \jmath^{p_0}.
    \end{equation}
    Now, suppose that $\gamma \delta$ was not locally nilpotent. 
    Then, there exists some $q \in P$ and $v \in C^q_m$ such that $(\gamma \delta)^n(v)\neq 0$ for all $n\in \bN$.
    Consequently, for $n\in \bN$, by \cref{eq:locally-nilpotent}, there exists strictly decreasing chains $p_n < \cdots < p_1 < p_0=q$.
    Now, since $(\gamma\delta)^{n+1}(v)\neq 0$, there is at least one of these chains that increases by one smallest term, obtaining the chain $p_{n+1}<p_n < \cdots < p_1 < p_0=q$. Inductively, we obtain an infinite decreasing chain in $P$, however, this contradicts the hypotheses that $P$ is a well-founded poset.
    Thus $\gamma\delta$ is locally nilpotent.  
\end{proof}
Consequently, \Cref{lemma:perturbation} implies upon perturbing $(C,\widehat{d})$ by $\delta$ to obtain $(C,d)$, we obtain a con\-trac\-tion
\begin{equation}
    \begin{tikzcd}[ampersand replacement=\&]
    (\oC, \od) \ar[r, "\bar{\beta}", shift right=0.3em, swap] \&
    (C, d)
    \arrow[l, "\bar{\alpha}", shift right=0.3em, swap]
    \arrow["\bar{\gamma}", from=1-2, to=1-2, loop, in=345, out=15, distance=10mm]
\end{tikzcd} \label{eq:contraction_conley}
\end{equation}
where $\cS = \delta(1- \gamma\delta)^{-1}$, and 
\begin{align}
    \od & = \alpha \cS \beta, & 
    \bar{\alpha} &= \alpha + \alpha \cS \gamma, & 
    \bar{\beta} &= \beta + \gamma \cS \beta,  & \bar{\gamma} &= \gamma +  \gamma \cS \gamma . \label{eq:dM}
\end{align}
We now verify that the contraction is $P$-filtered, and the induced differential $\od$ is strict. 

\begin{lemma}\label{lem:contraction_Pfiltered} The contraction of $(C,d)$ obtained by perturbing $(C,\widehat{d})$ by $\delta$ is $P$-filtered.
\end{lemma}

\begin{proof}
We first check that $\cS = \delta(1- \gamma\delta)^{-1}$ is $P$-filtered by writing it as a sum $\cS = \delta \sum_{i = 0}^\infty(\gamma \delta)^i$, which is well-defined since we have shown in \Cref{lem:pgraded-locallynilp} that $\gamma\delta$ is locally nilpotent. By construction, both $\gamma$ and $\delta$ are $P$-filtered (see \cref{eq:abc_proxy,eq:delta_graded}). Since sums, powers and compositions of $P$-filtered maps are $P$-filtered, we thus conclude that $\cS$ is $P$-filtered. 

Since  $\alpha, \beta, \gamma$ are also $P$-filtered  by construction (\cref{eq:abc_proxy}), 
$\bar{\alpha},\bar{\beta},\bar{\gamma},\od$ are all $P$-filtered. 
\end{proof}

\begin{lemma} \label{lem:connection_strict} The differential $\od= \alpha \cS \beta$ obtained by perturbing $(C,\widehat{d})$ by $\delta$ is strict.
\end{lemma}
\begin{proof}
    Because we have shown above that $\cS$ is $P$-filtered, and $\alpha, \beta$ are $P$-filtered by construction, 
    $\od^p = \alpha^p \cS^p \beta^p$.
    By construction $\cS^p=0$ for all $p \in P$ since $\delta^p=0$.
\end{proof} 

We have now assembled the relevant material to prove our main theorem.
\mainthm*
\begin{proof}
  Consider the perturbation $\delta$ \cref{eq:delta_graded} to the contraction in \cref{eq:Pgraded_elementary_contraction}:
  \[    \begin{tikzcd}[ampersand replacement=\&]
    (\oC, 0) \ar[r, "\beta", shift right=0.3em, swap] \&
    (C, \widehat{d})
    \arrow[l, "\alpha", shift right=0.3em, swap]
    \arrow["\gamma", from=1-2, to=1-2, loop, in=345, out=15, distance=10mm]
\end{tikzcd}.\]
  Since $P$ is well-founded, the chain homotopy $\gamma$ and perturbation $\delta$ are such that $\gamma \delta$ are locally nilpotent (\Cref{lem:pgraded-locallynilp}), the conditions of the perturbation lemma (\Cref{lemma:perturbation}) are satisfied for the perturbation $\delta$ to the contraction in  \cref{eq:Pgraded_elementary_contraction}. By \Cref{lem:contraction_Pfiltered}, the contraction of $(C,d)$ onto $(\oC, \od)$ (\cref{eq:contraction_conley}) obtained via the application of the perturbation lemma is $P$-filtered. Thus, the contraction satisfies condition \labelcref{M1} in the definition of the Conley complex (\Cref{def:connection_matrix}). Furthermore, we have shown that the differential $\od$ of the contracted complex is strict (\Cref{lem:connection_strict}), satisfying condition \labelcref{M2}. Thus, $(\oC, \od)$ is a Conley complex of $(C,d)$. 
\end{proof}

\begin{remark} \label{rmk:novelty}
We briefly remark on how our approach contrasts with that of~\cite[Theorem~8.1]{Robbin1992LyapunovFunctor}, where they prove the existence of the Conley complex for a $P$-graded chain complex $(C,d)$ implicitly constructing a contraction by induction (see also~\cite[Theorem 5.3.2]{Mrozek2025ConnectionDynamics} for a contemporary account of the proof). Similar to their approach, our result~\Cref{thm:main} also relies on homology decompositions. 
Robbin and Salamon~\cite{Robbin1992LyapunovFunctor,Mrozek2025ConnectionDynamics} 
showed that there exists a $P$-filtered subcomplex $(M, d\rvert_M)$ of $C$ which is a Conley complex for $C$. 
Our perturbative derivation can be thought of as being `dual' to their approach: our chain groups are the naive choice of relative homology subgroups $\bigoplus_{p \in P} H_\bullet(C^p)$, and we derive a new differential $\od$ from $(C,d)$ using the perturbation lemma (\Cref{lemma:perturbation}). 
\end{remark}

\subsection{The Conley Complex from Morse Matchings} \label{ssec:conley_amt}
We can also arrive at our main result~\Cref{thm:main} for $P$-graded chain complexes via the formalism of algebraic Morse theory for based complexes, as outlined in~\Cref{ssec:AMT}. We do so by expressing $P$-graded complex $(C,d)$ equipped with splittings of relative chains as a based complex in the fashion of \Cref{ex:Pfiltbased}, and defining a Morse matching. This is essentially the same derivation of the Conley complex as that in~\Cref{thm:main}, since the theorem that we rely on (\Cref{thm:skoldberg}) from~\cite{Skoldberg2006MorseViewpoint, Skoldberg2018} is itself a consequence of the homological perturbation theory (\Cref{lemma:perturbation}) which we apply in \Cref{ssec:relhomperturb} to arrive at a proof of~\Cref{thm:main}. Nonetheless, we present this perspective as it gives a conceptual link to discrete Morse theory, in particular the approach in~\cite{harker2021} where instead of using an algebraic Morse matching, they obtain the connection matrix via applying acyclic partial matchings on the underlying cell complex iteratively. 

Following the outline of algebraic Morse theory in~\Cref{ssec:AMT}, we first describe a $P$-graded chain complex, equipped with homology decompositions of relative homology groups $C^p$, as a based complex; this builds on our previous example~\Cref{ex:homology_decomposition_based} for an `ungraded' chain complex equipped with a homology decomposition. 

\begin{example} \label{ex:Pfiltbased}
    For a $P$-graded chain complex $C$ with a homology decomposition $\zeta^p_n: C^p_n \xrightarrow{\cong} H^p_n \oplus B^p_n \oplus K^p_n$ of relative chain groups, we can express $C$ as a based complex by writing the chain groups as 
    \[C_n =  \bigoplus_{p \in P} C^p_n \cong  \bigoplus_{p \in P} H^p_n \oplus B^p_n \oplus K^p_n.\] 
    In this case, the index set for each dimension is $I_n =P \times \{H_n, B_n,K_n
    \} $, if we write $C_n^{(p, V)} := V_n^p$ for $V \in \{H, B,K
    \}$. We can write the projection and inclusion maps as the composition 
    \begin{align*}
        \jmath_n^{(p,V)} &: C_n \overset{\jmath^p_n}{\twoheadrightarrow} C^p_n \xrightarrow[\cong]{\zeta^p_n} H^p_n \oplus B^p_n \oplus K^p_n \twoheadrightarrow V^p_n  \\ \imath_n^{(p,V)}&: V^p_n \hookrightarrow H^p_n \oplus B^p_n \oplus K^p_n \xrightarrow[\cong]{(\zeta^p_n)^{-1}}C^p_n  \xhookrightarrow{\imath^p_n} C_n .
    \end{align*}
    For a pair of poset elements $p \leq q$ and $V,W \in \{H, B, K\}$, the homology decomposition of relative chain groups leads to the following restricted differentials:
\[
\begin{tikzcd}[column sep=1.5cm,
/tikz/column 2/.style={column sep=-0.5em},
/tikz/column 3/.style={column sep=-0.5em},
/tikz/column 4/.style={column sep=-0.5em},
/tikz/column 5/.style={column sep=-0.5em}]
    V^q_n \ar[d, "d_n^{((p,W)(q,V))}", swap] 
    \ar[r, hookrightarrow, "\imath_n^{(q,V)}"] 
    &
    C_n \ar[d, "d"]
    \\
    W^p_{n-1} &
    C_{n-1} \ar[l,twoheadrightarrow, "\jmath^{(p,W)}_{n-1}"] 
\end{tikzcd}
\]
In particular, consider the case where $p=q$. 
Given $V,W \in \{H,B,K\}$, we can write  
\begin{equation*}
    d_n^{((p,W)(p,V))}: V^p_n \hookrightarrow H^p_n \oplus B^p_n \oplus K^p_n \xrightarrow[\cong]{(\zeta^p_n)^{-1}} C^p_n \xrightarrow{t^{p}_n} C^p_{n-1}  \xrightarrow[\cong]{\zeta^p_{n-1}} H^p_{n-1} \oplus B^p_{n-1} \oplus K^p_{n-1} \twoheadrightarrow W^p_{n-1}.
\end{equation*}
From \Cref{ex:homology_decomposition_based}, 
we deduce that $d_n^{((p,W)(p,V))}\neq 0$ only 
when $V = K$ and $W = B$, in which case $d_n^{((p,B)(p,K))}: K_n^p = B_{n-1}^p$ is the identity map.
\end{example}

Having formulated a $P$-graded chain complex, equipped with homology decompositions of relative chains, as a based complex, we can now define a Morse matching, generalising the ungraded case in~\Cref{ex:homology_decomposition_based_2}.

\begin{lemma} \label{lem:morse_matching_connection_matrix}
    Assume $P$ is a well-founded poset. Given a $P$-graded chain complex of vector spaces $(C,d)$ with homology decomposition of relative chains $\zeta^p$, consider its expression as a based complex with index set $I_n = P \times \{H_n, B_n,K_n\}$ (\Cref{ex:Pfiltbased}):
    \[C_n \cong \bigoplus_p H^p_n \oplus B^p_n \oplus K^p_n.\]
    Then the set of directed edges $\sM = \{(p,K_n) \rightharpoonup (p, B_n) \ : \ p \in P, n \in \Z\}$ is a Morse matching, and the critical set of $\sM$ is given in each dimension $n$ by $\sM^0_n = \{ (p,H_n) \ : \ p \in P\}$ (i.e. $\widetilde{C}_n = \bigoplus_{p \in P} H_n(C^p)$). 
\end{lemma}
\begin{proof}
We check that $\sM$ satisfies the two conditions for a Morse matching as described in \Cref{def:Morse-matching}. 
First, we recall from \Cref{ex:Pfiltbased} that for all $(p,K_n) \rightharpoonup (p, B_{n-1}) \in \sM$ there is an isomorphism
\[
d^{((p,B)(p,K))}_n = \jmath^{(p,B)}_n d_n \imath^{(p,K)}_n: K_n^p \xrightarrow{=} B^p_{n-1}  \ .
\]
Thus the first condition is satisfied. 
For the second condition, we need to find a partial order $\preceq_n$ on $I_n$, such that the source and targets of length two directed paths on $G(C)^\sM$ are ordered by $\preceq_n$. We first note that all directed edges in $G(C)^\sM$ are either of the form $(p, B_{n-1})\rightharpoondown (p,K_n)$ or $(q,V_n) \rightharpoonup (r,W_{n-1})$ for $q\gneq r$. Hence, all length two directed paths on $G(C)^\sM$ with distinct source and target in $I_n$ are of either of the following forms, for $p \lneq q$:
\begin{align*}
    (q, V_n) &\rightharpoonup (p, B_{n-1}) \rightharpoondown (p, K_{n}) \qc{or}\\
    (q, B_n) &\rightharpoondown (q, K_{n+1}) \rightharpoonup (p, W_{n}).
\end{align*}
Consider then the following partial order: $(p, W_n) \preceq_n (q,V_n)$ only if they are identical, or $p \lneq q$.
Since $\leq$ is a partial order on $P$, it is straightforward to check that $\preceq_n$ is a partial order on $I_n$, i.e. reflexive, anti-symmetric, and transitive. We can also verify by inspection that the source and target of any length two path is ordered by $\precneq_n$. The partial order $\preceq_n$ is also well-founded: any descending sequence of unique elements in $I_n$ must decrease its $P$-grading, and the corresponding descending sequence of unique poset elements must be finite by $P$ being well-founded. Having satisfied both conditions in~\Cref{def:Morse-matching}, we have hence shown that $\sM$ is a Morse matching.

For the final claim, since all the relative pre-boundaries and boundaries are matched by $\sM$, what remains are the relative homology groups $H^p_n$, and thus $\widetilde{C}_n= \bigoplus_{p \in P} H^p_n$.
\end{proof}
\begin{remark}
    Notice that, if $P$ is a finite poset, then \Cref{lemma:morse-cycles-matching} applies and one can alternatively show that $\sM$ is a Morse matching by the fact that $G(C)^\sM$ is acyclic.
\end{remark}
We can now consider the contraction of $(C,d)$ induced by the Morse matching in~\Cref{lem:morse_matching_connection_matrix}, according to~\Cref{thm:skoldberg}. The proposition below states that the Morse matching due to homology decompositions of relative chains induces the same contraction as that in \Cref{thm:main}. Thus this Morse matching yields a compression of a $P$-graded chain complex into its Conley complex, analogous to how acyclic partial matchings on a cell complex reduces the cell complex to an algebraic Morse complex. 

\begin{proposition} \label{prop:equivalence}
    Assume $P$ is a well-founded poset, and $(C,d)$ is a $P$-graded chain complex of vector spaces with homology decomposition of relative chains $\zeta^p$. Given the Morse matching $\sM$ on $(C,d)$ in \Cref{lem:morse_matching_connection_matrix}, the corresponding contraction of $(C,d)$ as a consequence of \Cref{thm:skoldberg} is identical to the $P$-filtered contraction as described in \Cref{thm:main}. 
\end{proposition}
\begin{proof}
We recall that the contraction in \Cref{thm:skoldberg} was derived by perturbing the contraction determined by a Morse matching $\sM$ (\cref{eq:based_elementary_contr_maps,eq:basecontraction_morsematching})
\begin{equation} \label{eq:contraction_star}
    \begin{tikzcd}[ampersand replacement=\&]
    (\widetilde{C}, 0) \ar[r, "g", shift right=0.3em, swap] \&
    (C,d-\tau) 
    \arrow[l, "f", shift right=0.3em, swap]
     \arrow["h", from=1-2, to=1-2, loop, in=345, out=15, distance=10mm]
\end{tikzcd},\tag{$\ast$}
\end{equation}
by $\tau$. To prove the statement in the proposition, it suffices to show two things: that 
\eqref{eq:contraction_star} is identical to the `pre-perturbation' contraction (\cref{eq:Pgraded_elementary_contraction,eq:abc_proxy}) as employed in the proof of \Cref{thm:main} in \Cref{ssec:relhomperturb}
\begin{equation} \label{eq:contraction_dagger}
    \begin{tikzcd}[ampersand replacement=\&]
    (\oC, 0) \ar[r, "\beta", shift right=0.3em, swap] \&
    (C, d-\delta)
    \arrow[l, "\alpha", shift right=0.3em, swap]
    \arrow["\gamma", from=1-2, to=1-2, loop, in=345, out=15, distance=10mm]
\end{tikzcd};\tag{$\dagger$}
\end{equation}
 and the respective perturbations $\tau, \delta$ are identical. Then, the contraction of the full complex $(C,d)$ obtained by perturbing \eqref{eq:contraction_star} by $\tau$ is identical to that obtained in \Cref{thm:main} by perturbing \eqref{eq:contraction_dagger} by $\delta$.

 We have shown in \Cref{lem:morse_matching_connection_matrix} that $\widetilde{C} = \oC$ as (dimension) graded vector spaces, and thus $(\widetilde{C},0) = (\oC,0)$  as chain complexes; all that remains is to show that the perturbations and the maps in the contraction are identical. In other words, that is, $\tau = \delta$, and  $(f,g,h) = (\alpha, \beta, \gamma)$.  Let us construct the maps $f,g,h$ and $\tau$ following the prescription in~\cref{eq:based_elementary_contr_maps}. First, the maps $f,g$ are simply the projection and inclusion maps $\alpha, \beta$ in \cref{eq:abc_proxy}. Furthermore, the map $h$ in \cref{eq:based_elementary_contr_maps} is equivalent to $\gamma$ in \cref{eq:abc_proxy} by the observation in \Cref{ex:homology_decomposition_based_4} that in the ungraded case, $\imath^{(K)}_{n+1} \big(d^{(BK)}_{n+1} \big)^{-1}\jmath^{(B)}_n = -\gamma_n$:
\begin{align*}
    h_n &= - \sum_{p \in P} \imath^{(p,K)}_{n+1} \big(d^{((p,B) (p,K))}_{n+1} \big)^{-1}\jmath^{(p,B)}_n = \sum_{p \in P} \imath^p_{n+1} \gamma^p_n  \jmath^p_n = \gamma_n.
\end{align*}
Then, the homology decomposition (\cref{eq:homology-decomposition}) also implies 
\begin{align*}
    \tau &= \sum_{(b\rightharpoonup a) \notin \sM}\imath^{(a)} \circ d^{(ab)} 
    \circ \jmath^{(b)}  = d - \sum_{(b\rightharpoonup a) \in \sM}\imath^{(a)} \circ d^{(ab)} 
    \circ \jmath^{(b)} = d- \sum_{p \in P } \imath^{(p,B)} d^{((p,B) (p,K))} \jmath^{(p,K)} \\
    &= d - \sum_{p \in P}\imath^p t^p \jmath^p = \sum_{p \lneq q}\imath^p t^{pq} \jmath^q = \delta. 
\end{align*}
The fourth equality follows from the description of $d^{((p,B) (p,K))}$ in~\Cref{ex:Pfiltbased}.
\end{proof}

\label{subsec:contraction-via-perturbation}
\section{Connection Matrix Algorithm}\label{sec:connection-matrix-algorithm}

As we have seen in~\cref{thm:main} choosing homology decompositions for each poset grade $\zeta^p_n\colon C^p_n \cong H^p_n \oplus B^p_n \oplus K^p_n$ gives rise to a $P$-filtered contraction between $(C,d)$ and $(\oC, \od)$.
A matrix representation of $\od$ then leads to a connection matrix $\Delta$. 
In this section, we develop algorithms to find a suitable splitting and  compute connection matrices. 
The full connection matrix algorithm is presented in~\Cref{alg:connection-matrix}. 
The algorithm first involves a routine (\Cref{alg:separating-bases}) that effects the homology decomposition $\zeta_p$ of relative chain groups $C^p$ in a computationally favourable basis, which we call a separating basis (\Cref{def:sep_basis}). We give an account of this in~\Cref{subsec:splittings-clearing}. We then give a description of the connection matrix $\Delta$ for $\od$ in the separating bases in~\Cref{subsec:conley_formulae}; in particular, we describe the matrix entries $\Delta_{ij}$ explicitly in~\Cref{cor:entry-connection}. Finally, we then describe in~\Cref{alg:connection-matrix} how the separating bases can be leveraged to perform a sequence of row and column operations that reduces a boundary matrix representation of $d$ in any basis to a connection matrix $\Delta$. We prove in~\Cref{prop:connection-matrix-algo} the correctness of the algorithm: $\Delta$ as the output of~\Cref{alg:connection-matrix} represents the Conley complex differential $\od$ in a basis consisting of relative homology cycles at each poset grade. 

Henceforth, we assume all chain complexes are finite dimensional (and hence bounded in dimension).
\begin{remark}
    In the next discussion on matrix algorithms, we make use of the following notation:
    \begin{itemize}
        \item We use fraktur notation $\fA$ to denote an explicit basis of a vector (sub)space. Given a basis $\fA$, we let $\langle \fA \rangle$ denote the vector subspace spanned by $\fA$.
        \item  If $M$ is a matrix representing a linear map $f: V \to W$ in some explicit bases $\mathfrak{V}$, $\mathfrak{W}$ for $V,W$ respectively, we write $c(M)$ for the 
        set of nonzero columns of $M$; that is, those basis vectors $u \in \mathfrak{V}$ such that $f(u) \neq 0$.
        \item We implicitly assume a total ordering of basis vectors for any given basis of a vector space, and index all matrices with that total ordering. For any vector $v=\sum_{u \in \fA}v_u u$ written in a basis $\fA$ with coefficients $v_u$, we denote the \emph{pivot} $\rho(v)$ to be the maximal $u \in \fA$ such that $v_u \neq 0$. Given a subset of vectors~$W$, we use the notation $\rho(W)=\{\rho(v) \mid v \in W\}$ to denote the associated collection of pivots.
        \item For a chain complex with differential $d_n$, we use $D_n$ to denote its matrix representation with respect to a chosen set of bases $\{\fA_i\}$ of chain groups $C_i$.  
    \end{itemize}
\end{remark}

\subsection{Splittings via the clearing optimisation}\label{subsec:splittings-clearing}
Suppose we are given boundary matrices $D_n$ representing $d_n$ in bases $\fA_n$ of $C_n$.
In this section, we describe an algorithm that computes split isomorphisms 
$
\zeta_n: C_n\cong H_n\oplus B_n \oplus K_n
$
(\cref{eq:homology-decomposition}) for all $n$. 
In other words, we seek coordinate transformations from an initial basis $\fA_n$ of $C_n$ to another basis $\fH_n\sqcup \fB_n \sqcup \fK_n$ of $C_n$, such that we can write the differential in the new basis as a block matrix 
\[
\begin{tikzpicture}[scale=1]

\draw[dashed] (0,0) rectangle (3,3);

\draw[dashed] (1,0) -- (1,3);
\draw[dashed] (2,0) -- (2,3);

\draw[dashed] (0,1) -- (3,1);
\draw[dashed] (0,2) -- (3,2);

\node at (0.5,3.3) {$\fH_n$};
\node at (1.5,3.3) {$\fB_n$};
\node at (2.5,3.3) {$\fK_n$};

\node at (-0.5,2.5) {$\fH_{n-1}$};
\node at (-0.5,1.5) {$\fB_{n-1}$};
\node at (-0.5,0.5) {$\fK_{n-1}$};

\node at (2.5,1.5) {$I$};
\node at (-1.5, 1.5) {$d_n = $};
\end{tikzpicture}.
\]
More formally, we desire a basis $\fH_n\sqcup \fB_n \sqcup \fK_n$ for each $C_n$, such that:
\begin{enumerate}
    \item $\fB_n$ is a basis of boundary chains: $\B_n := \langle \fB_n \rangle = \ima (d) =: B_{n-1}$;
    \item $\fH_n$ is a homology basis: $\ker d = \langle \fH_n \sqcup \fB_n \rangle$; in other words, $\cH_n := \langle \fH_n \rangle  \cong H_n$;
    \item $\fK_n$ is a pre-boundary basis: for $\fK_n = \{u_1,\ldots, u_m\}$, we have $\fB_{n-1} = \{d(u_1),\ldots, d(u_m)\}$. This implies $d$ restricted to $\cK_n := \langle \fK_n \rangle$ is an isomorphism from $\cK_n$ to $\B_{n-1} = B_{n-1} =: K_n$. 
\end{enumerate}

Next, we present an algorithm that obtains a `separating basis' (see \Cref{def:sep_basis}), which enhances the efficiency in the subsequent computations in the main connection matrix algorithm (\Cref{lem:properties-separating-bases,prop:canonical-differential}).

Our approach is based on `killing reduction'~\cite{Chen2011}, which we call \emph{clearing optimisation} following~\cite{Bauer2014ClearChunks}. 
We refer the reader to those references for the correctness of the algorithm, and give a description here.  
The algorithm, detailed in~\cref{alg:separating-bases}, proceeds sequentially starting from the top dimensional boundary matrices. We give an outline of the algorithm:
\begin{enumerate}
    \item At each step, the algorithm performs a column reduction $R_n = D_n T_n$ of $D_n$ by means of the clearing optimisation, assuming that $D_{n+1}$ has also been similarly reduced in the previous step, and leveraging data from $R_{n+1}$ to avoid unnecessary column operations.
    \item From the reduction, we then read off the bases $\fH_n, \fK_n \subset C_n$ from columns of $T_n$ and  $\fB_{n-1} \subset C_{n-1}$ from columns of $R_n$.
\end{enumerate}
Since we assume that $C$ is bounded, there exists a largest integer $N >0$ such that $D_N\neq 0$. Because $D_{N+1}=0$, any reduction of $D_{N+1}$ must be trivial and $R_{N+1}=0$. If all differentials are trivial, the splitting is immediate.
In each reduction step, the clearing optimisation procedure~(\cref{alg:clear-reduce}) ensures $R_n$ is column reduced,  
$T_n$ is square upper diagonal with ones down the diagonal, and $c(R_{n+1}) \subset c(T_n)$. 
In particular, the columns from $T_n$ are linearly independent and a basis for $C_n$. 
The desired bases are then obtained by the following partition of column vectors $c(T_n) = \fH_n \sqcup \fB_n \sqcup \fK_n$:
\begin{align}
\fK_n & = \{ v  \in c(T_n) \mbox{ such that } D_n v \neq 0\}; \\
   \fH_n & = \{ v  \in c(T_n) \setminus c(R_{n+1}) \mbox{ such that } D_n v = 0 \}; \quad \text{and}\\
   \fB_{n} & =c(R_{n+1}).
\end{align}
Note that $d(\fK_n) = c(R_n) = \fB_{n-1}$, and thus $d$ restricted to $\langle \fK_n \rangle$ is an isomorphism onto $\langle c(R_n) \rangle = \langle \fB_{n-1} \rangle$; the non-zero columns of $R_{n+1}$, i.e. $\fB_n$ is a basis for $B_n$; and $\ker d = \langle \fH_n \sqcup \fB_n \rangle$. It follows from the properties above that this choice of basis gives an explicit matrix representation of the splitting isomorphisms \cref{eq:homology-decomposition}:
\begin{equation}\label{eq:direct-sum-split-bases}
\zeta_n: C_n = \langle c(T_n)\rangle = 
\langle \fH_n\rangle \oplus \langle \fB_n\rangle \oplus \langle \fK_n\rangle = 
\cH_n \oplus \B_n \oplus \cK_n \xrightarrow{\cong} H_n \oplus B_n \oplus K_n
.
\end{equation}

\Cref{alg:clear-reduce} is presented for a matrix $D$ and a set $S$, which, in our present case, would correspond to the pair $(D_n,c(R_{n+1}))$. 
The clearing optimisation reduces column operations by `pre-computing' the columns of $T_n$ with information from $R_{n+1}$. 
Basically, each column $b \in c(R_{n+1})$ gives the $\rho(b)$ column in $T_n$. In addition, for each $b \in \fB_n$, since $D_nb=0$, one can set the $\rho(b)$-column from $D_n$ to zero, before proceeding with the usual column reduction. Thus, this procedure enforces $c(R_{n+1})\subset c(T_n)$. 

\begin{algorithm}
\caption{\texttt{clear\_reduce}}\label{alg:clear-reduce}
\KwData{\begin{enumerate} 
\item Matrix $D$ of size $M\times N$; and 
\item a set $S$ of linearly independent vectors such that $\langle S\rangle \subset \ker(D)$.
\end{enumerate}}
\KwResult{Matrices $R, T$ with $R = D T$, $R$ column reduced, $T$ upper-triangular and $S\subset c(T)$.}
 $T \gets \id_{N}$ \;
 $R \gets D$ \;
 \For{$b \in S$}{
    Set $\rho(b)$-column from $T$ equal to $b$ \;
    Set $\rho(b)$-column from $R$ equal to $0$ \;
 }
 \For{$i \gets M$ \textbf{down to} $1$}{
   $Q \gets  
   \{1 \leq j \leq N \mid \rho(R[j])=i\}$ \;
   \If{$Q\neq \emptyset$}{
   $m \gets \min (Q)$ \;
   \For{$j\in Q\setminus\{m\}$}{
        $a \gets R[i,j]  / R[i,m]$ \;
        $R[:,j] \gets R[:,j] - a\cdot R_n[:,m]$ \;
        $T[:,j] \gets T[:,j] - a\cdot T[:,m]$ \;
    }
   }
 }
\Return $R, T$
\end{algorithm}

\begin{algorithm}
\caption{\texttt{separating\_via\_clearing}}\label{alg:separating-bases}
\KwData{Sequence of matrices $(D_n)_{n=0}^N$ with $D_i D_{i+1}=0$ for all $0 \leq i < N$ and $D_0=0$.}
\KwResult{Separating bases $((\fH_n, \fB_n, \fK_n))_{n=0}^N$, matrices $(T_n)_{n=0}^N$ returned by \texttt{clear\_reduce} }
 $\fB_{N} \gets \emptyset$\;
 \For{$n \gets N$ \textbf{ down to } $0$}{
  $R_n, T_n \gets \texttt{clear\_reduce}(D_n, \fB_n)$ \;
  $\fH_n \gets \{ v \in c(T_n) \setminus \fB_n \mbox{ such that } D_nv=0\}$ \;
  \If{$n>0$}{
   $\fK_n \gets \{ v  \in c(T_n) \mbox{ such that } D_n v \neq 0\}$ \;
   $\fB_{n-1} \gets c(R_n)$ \;
   }
}
$\fK_0 \gets \emptyset$ \;
\Return $((\fH_n, \fB_n, \fK_n))_{n=0}^N$, $(T_n)_{n=0}^N$ 
\end{algorithm}
Furthermore, in \Cref{lem:separating-basis}, we see that $(\fH_n, \fB_n, \fK_n)$ is a separating basis for $C_n$ in the sense of the following definition:
\begin{definition} \label{def:sep_basis}
The triple $(\fH_n, \fB_n, \fK_n)$ forms a \emph{separating basis} for $C_n$ (with respect to $\fA_n$) if the following conditions are satisfied:
\begin{enumerate}[label=(S\arabic*)]
    \item\label[condition]{item:separating-first} for all elements $z \in \fH_n$ and all $w \in \fK_n$ it follows that $z_\beta =w_\beta=0$ for all $\beta \in \rho(\fB_n)$, 
    \item\label[condition]{item:separating-second} for fixed $z \in \fH_n$, it follows that, for all $v \in \fH_n \setminus \{z\}$ and all $w \in \fK_n$, 
    $v_{\rho(z)}=w_{\rho(z)}=0$, and
    \item\label[condition]{item:separating-ones} for all $z \in \fH_n$ we have $Z_{\rho(z)}
    =1_k$.
\end{enumerate}
\end{definition}
Altogether, given a sequence of differentials $(D_n)_{n=0}^N$, one obtains, by the aforementioned procedure, a sequence of separating bases $((\fH_n, \fB_n, \fK_n))_{n=0}^N$; this is outlined in \Cref{alg:separating-bases}. 

\begin{lemma}\label{lem:separating-basis}
    $(\fH_n, \fB_n, \fK_n)$ is a separating basis for $C_n$ (w.r.t. $\fA_n$).
\end{lemma}
\begin{proof}
   Recall the description of the clearing optimisation outlined in \Cref{alg:clear-reduce}. 
    In particular, recall that all columns from $D_n$ indexed by $\rho(\fB_n)$ are set to zero before proceeding with the column reductions. 
    One can then assume that \Cref{item:separating-first} is satisfied, since performing column operations in $D_n$, all columns from $\rho(\fB_n)$ were zero. 
    We can even refine further.
    Indeed, since each $z \in \fH_n$ is such that $D_nz=0$, then, if we followed the reduction of $D_n$ using \Cref{alg:clear-reduce}, one never adds the column labelled by $\rho(z)$ from $D_n$ to columns on its right, since it becomes trivial. 
    Thus, \Cref{item:separating-second} follows.
    In addition, when performing left-to right column additions in \Cref{alg:clear-reduce}, the target column is not multiplied by any scalar, and so \Cref{item:separating-ones} automatically holds.
\end{proof}

For convenience, for any $V \in \{ \cH, \B, \cK\}$ we write $V\coloneqq\bigoplus_{n\in \Z} V_n$.
Similarly, given $\fD \in \{ \fH, \fB, \fK\}$, we write $\fD\coloneqq \bigsqcup_{n\in \Z} \fD_n$ and $\zeta(\fD) = \{ \zeta(z) \colon z \in \fD\}$.
By construction, \Cref{alg:separating-bases} leads to separating bases $((\fH_n, \fB_n, \fK_n))_{n=0}^N$ for $C_n$ (w.r.t. $\fA_n$) for all $0 \leq n \leq N+1$.
Choosing trivial bases for other degrees, we obtain a separating basis $(\fH, \fB, \fK)$ for $C$ (w.r.t. $\fA$).

Separating bases have properties that render them computationally convenient: 
given a subset of basis vectors $\fD \subset \fA$, we define the projection $\jmath_{\fD}\colon C \rightarrow \langle \fD\rangle$ which is given by sending 
$
v=\sum_{\alpha\in \fA} v_\alpha \alpha 
$
to 
$
\sum_{\alpha\in \fD} v_\alpha \alpha,
$
for all $v \in C$.
In addition, we also consider projections (not necessarily chain maps) $\jmath_V\colon C \rightarrow V$ for $V\in \{H, B, K\}$, derived from the isomorphisms $\zeta\colon C\cong H \oplus B \oplus K$ (\cref{eq:direct-sum-split-bases}). 
Lastly, we introduce the map $\theta \colon \langle \rho(\fH)\rangle \rightarrow H$ given by $\theta(\rho(z)) = \zeta(z)$ for all $z \in \fH$; this is an isomorphism since the pivots from $\fH$ are all distinct.
\begin{lemma}\label{lem:properties-separating-bases}
    Let $(\fH, \fB, \fK)$ be a separating basis for $C$.
    Then, 
    \begin{enumerate}[label=(\roman*)]
        \item\label[property]{property:separating-base-first} $\ker(\jmath_{\rho(\fB)})=\ker(\jmath_B) = \langle \fA \setminus \rho(\fB)\rangle$, and  
        \item\label[property]{property:separating-base-second} $(\jmath_{H})_{\mid \cH \oplus \cK}= (\theta \jmath_{\rho(\fH)})_{\mid \cH \oplus \cK}$.
    \end{enumerate}
\end{lemma}
\begin{proof}
    First, suppose that $v \in C$ is such that $\jmath_{\rho(\fB)}(v)=0$; in particular, $v \in \langle \fA \setminus \rho(\fB)\rangle$.
    On the other hand, \Cref{item:separating-first} implies that $\langle \fH\sqcup \fK\rangle \subset \langle \fA \setminus \rho(\fB)\rangle$.
    Which, together with 
    \[
    \fA \setminus \rho(\fB) = \rho\big(\mbox{$\bigsqcup_{n\in \Z} c(T_n)\big) \setminus \rho(\fB)$} = 
    \mbox{$\bigsqcup_{n\in \Z}\rho(c(T_n)\setminus \fB_n)$} = 
    \mbox{$\bigsqcup_{n\in \Z}\rho(\fH_n \sqcup \fK_n)$}  = 
    \rho(\fH \sqcup \fK),
    \]
    implies that $\langle \fH\sqcup \fK\rangle = \langle \fA \setminus \rho(\fB)\rangle$. Hence $v \in \langle \fA \setminus \rho(\fB)\rangle=\cH \oplus \cK$ and so, by \cref{eq:direct-sum-split-bases}, it follows that $\jmath_{B}(v)=0$. 
    Now, suppose that $v \in C$ is such that $\jmath_{B}(v)=0$.
    Then $v \in \cH\oplus \cK=\langle \fA \setminus \rho(\fB)\rangle$, and it follows that $\jmath_{\rho(\fB)}(v)=0$, so that \Cref{property:separating-base-first} holds. 

    Let us now show \Cref{property:separating-base-second}.
    For this, recall that, given $v \in \cH\oplus \cK$, one can write 
    $v=\sum_{z \in \fH}v_z z + \sum_{w \in \fK}v_w w$ for unique coefficients $v_z,v_w \in k$ for all $z \in \fH$ and $w \in \fK$.
    In particular, it follows that
    \[
    \jmath_{H}(v) = 
    \jmath_{H}\left( \sum_{z \in \fH}v_z z + \sum_{w \in \fK}v_w w\right) =
    \sum_{z \in \fH}v_z \zeta(z),
    \]
    where we have used $\jmath_{H}(z)=\zeta(z)$ for all $z \in \fH$ and $\jmath_{H}(w)=0$ for all $w \in \fK$.
    Similarly, one obtains 
    \[
    \theta\jmath_{\rho(\fH)}(v) = 
    \theta\jmath_{\rho(\fH)}\left( \sum_{z \in \fH}v_z z + \sum_{w \in \fK}v_w w\right) =
    \theta\left( \sum_{z \in \fH}v_z \rho(z)\right) =
    \sum_{z \in \fH}v_z \zeta(z).
    \]
    where we have used $\jmath_{\rho(\fH)}(z)=\rho(z)$ for all $z \in \fH$ and $\jmath_{\rho(\fH)}(w)=0$ for all $w \in \fK$, which, in turn, follow from \Cref{item:separating-second} and \Cref{item:separating-ones} of separating bases.
    Thus, \Cref{property:separating-base-second} holds.
\end{proof}

\subsection{The Conley differential and the Connection matrix under a Separating Bases}\label{subsec:conley_formulae}

In this subsection, we deduce a technical result concerning the Conley Complex for a given choice of separating bases for $C^p$ for all $p\in P$. 
Namely, in \Cref{prop:canonical-differential} we provide a formula for the Conley differential. 
This allows to explicitly express the entries of the Connection matrix for a particular choice of separating bases; see \Cref{cor:entry-connection}.

To start, we consider separating bases taken at the relative level; i.e. there exist separating bases $(\fH^p_n, \fB^p_n, \fK^p_n)$ such that there are relative chain decompositions  $C^p_n=\cH^p_n\oplus \B^p_n \oplus \cK^p_n$ with $\cH^p_n=\langle \fH^p_n\rangle$, $\B^p_n=\langle \fB^p_n\rangle$ and $\cK^p_n=\langle \fK^p_n\rangle$ for all $p \in P$ and $n\in \Z$, as in \cref{eq:direct-sum-split-bases}.

\begin{remark}\label{rmk:notation_2}
    For convenience, for any $V \in \{ \cH, \B, \cK\}$ we write $V\coloneqq\bigoplus_{p\in P, n\in \Z} V^p_n$.
    Similarly, given $\fD \in \{ \fH, \fB, \fK\}$, we write $\fD\coloneqq \bigsqcup_{p \in P, n\in \Z} \fD_n^p$ and 
    $\zeta(\fD) = \bigsqcup_{p \in P, n\in \Z} \zeta^p_n(\fD_n^p)$.
    Altogether, we have the following decomposition for $C$:
    \[
        C = \cH \oplus \B \oplus \cK = \langle\fH\rangle \oplus \langle \fB \rangle \oplus \langle \fK \rangle \cong H \oplus B \oplus K\ ,
    \]
    and we say that $(\fH, \fB, \fK)$ is a separating basis for $C$ (w.r.t. $\fA$). 
\end{remark}

\begin{lemma}\label{lem:canonical-chains}
$ \ima(1+\widehat{d}\gamma) = 
\ker(\gamma)= \cH\oplus \cK 
$.
\end{lemma}
\begin{proof}
    First, we have $\ima(1+\widehat{d}\gamma) \subset \ker(\gamma)$
    by the equality $\gamma + \gamma \widehat{d} \gamma=0$ which holds by applying the properties of chain contractions from \cref{eq:contraction_b} to the contraction in \cref{eq:Pgraded_elementary_contraction}.
    Next, the other inclusion $\ker(\gamma) \subset \ima(1+\widehat{d}\gamma)$ follows by $(1+\widehat{d}\gamma)(\ker(\gamma))=\ker(\gamma)$.
    Hence, $\ima(1+\widehat{d}\gamma)=\ker(\gamma)$ holds.

    On the other hand, 
    by the definition of $\gamma$ from \cref{eq:contraction_b}, $\gamma$ restricts to an isomorphism
    $\gamma_{|\B}\colon \B \cong B \xrightarrow{=}K\cong \cK$.
    Hence, it follows that $\cK=\ima(\gamma)$ and $\ker(\gamma) \subset \cH \oplus \cK$.
    In addition,
    the equality $\gamma^2=0$ implies that $\cK=\ima(\gamma) \subset \ker(\gamma)$. Also, $\cH=\ima(\beta)\subset \ker(\gamma)$ by $\gamma \beta=0$. Hence, $\cH \oplus \cK \subset \ker(\gamma)$.
\end{proof}

Next, we recall the contraction in \cref{eq:contraction_conley} which is given by
\begin{equation}\label{eq:contraction-repeated}
    \begin{tikzcd}[ampersand replacement=\&]
    (\oC, \od) \ar[r, "\widetilde{\beta}", shift right=0.3em, swap] \&
    (C, d)
    \arrow[l, "\widetilde{\alpha}", shift right=0.3em, swap]
    \arrow["\widetilde{\gamma}", from=1-2, to=1-2, loop, in=345, out=15, distance=10mm]
\end{tikzcd} 
\end{equation}
where
$\cS = \delta(1- \gamma\delta)^{-1}$ and $\od= \alpha \cS \beta$.
Recall that  $\oC_n =  \bigoplus_{p \in P} H^p_n$ for all $n\geq 0$. 
In this context $\alpha = \bigoplus_{p \in P}\jmath_{H^p}$ and $\beta = \bigoplus_{p \in P}\imath_{H^p}$.
Next, we introduce a useful description of the differential $\od$ of the Conley complex that holds when choosing separating bases.

\begin{proposition}\label{prop:canonical-differential}
    Let $(\fH^p, \fB^p, \fK^p)$ are be separating bases for $C^p$ for all $p \in P$.
    Suppose that we have a contraction as in \cref{eq:contraction-repeated}. 
    Then, $\od(z)=\theta\jmath_{\rho(\fH)}d(1 + \gamma\cS)\beta(z)$ for all $z \in \oC$. 
\end{proposition}
\begin{proof}
    To start, since $\cS=\delta(1-\gamma\delta)^{-1}$, we obtain 
    \[
    \cS - \delta\gamma \cS = 
    \delta (1-\gamma\delta)^{-1} - \delta \gamma \delta(1-\gamma\delta)^{-1} =
    \delta ( 1 - \delta \gamma )(1-\gamma\delta)^{-1} = \delta
    \]
    In particular, $\cS = \delta(1 + \gamma\cS)$.
    
    Next, since $\alpha$ is a chain map then $\alpha \widehat{d}=0\alpha = 0$, and it follows that $\alpha d = \alpha (\widehat{d} + \delta) = \alpha \delta$.
    Hence, 
    \begin{equation}\label{eq:pi-M-d-M}
          \od = 
          \alpha\cS\beta=
          \alpha \delta(1 + \gamma\cS) \beta =
          \alpha d(1+\gamma\cS)\beta\ . 
    \end{equation}
    On the other hand, we claim that $\ima(d(1+\gamma\cS)\beta) \subseteq \cH \oplus \cK$. 
    For this, recalling that  $\widehat{d}\beta=\beta0=0$, we obtain
    \[
    d(1+\gamma\cS)\beta=
    (\delta + \widehat{d})(1 + \gamma\cS) \beta = 
    \delta (1+\gamma \cS)\beta + \widehat{d} \gamma \cS \beta = \cS\beta + \widehat{d}\gamma \cS \beta = (1+\widehat{d}\gamma)\cS\beta\ ,
    \]
    which implies the claim by \Cref{lem:canonical-chains}, since $\ima(1+\widehat{d}\gamma)=\cH\oplus \cK$.
    
    To finish, given $z \in \oC$, we write $v = d(1+\gamma\cS)\beta(z)$ and, using \Cref{property:separating-base-second} from \Cref{lem:properties-separating-bases}, it follows that $\jmath_{H^p}(v^p)=\theta^p \jmath_{\rho(\fH^p)}(v^p)$ for all $p \in P$. 
    Altogether, noticing that 
    $\alpha = \jmath_H$ and using \cref{eq:pi-M-d-M}, we obtain
    \[
    \od(z) = \alpha (v) = 
    \jmath_H(v) =
    \theta \jmath_{\rho(\fH)}(v)
    \]
    Thus, restricting the column $v$ to the entries indexed by $\rho(\fH)$ leads to the coordinates of $\od(z)$ in terms of the image $\zeta(\fH)$; which is a basis for $\oC$.
\end{proof}

\begin{remark} 
    A consequence of \Cref{prop:canonical-differential} is that given $z \in H$, the image under the Conley differential $\od(z)$ corresponds to a chain  $d(1+\gamma\cS)\beta(z) \in \cH\oplus \cK$. 
    This construction adapts work from Section~3 in~\cite{Harker2014DiscreteMaps} to the context of $P$-graded complexes. In Definition~3.8. from~\cite{Harker2014DiscreteMaps} pre-boundaries are called \emph{$\cK$-chains} and elements from $\cH \oplus \cK$ are called \emph{canonical chains}. 
\end{remark}

We wish to compute the connection matrix $\Delta$ in the fixed basis $\zeta(\fH)$ for the $P$-graded space $H$. 
To do this, for each $z \in \fH$, we need to compute the coefficients of $\od(z)$ in terms of $\zeta(\fH)$.
The following result summarises this expression. It is an immediate consequence from \Cref{prop:canonical-differential}.
\begin{corollary}\label{cor:entry-connection}
    Let $(\fH^p, \fB^p, \fK^p)$ be a separating basis for $C^p$ for all $p \in P$.
    Denote by $\Delta$ the connection matrix; the matrix associated to $\od$ with respect to the basis $\zeta(\fH)$.
    Given two generators $z_i, z_j \in \zeta(\fH)$, denote by $\Delta_{i,j}$ the entry of $\Delta$ from the row $i$ and the column $j$ (i.e. associated to $z_i$ and $z_j$ respectively). Then, 
    \[
    \Delta_{i,j} = \theta \jmath_{\rho(z_i)}d(1+\gamma\cS)\beta(z_j)\ .
    \]
\end{corollary}

\subsection{Connection matrix algorithm}
\label{subsec:algorithm}

In this subsection, we present a simple procedure to obtain a connection matrix. 
Recall that $C_n = \bigoplus_{p \in P}\C^p_n$ and suppose that we have fixed bases $\fA^p_n$ for each term $\C^p_n$. 
Putting together these bases we obtain a basis $\fA$ for $C$.
Now, let $D$ be the boundary matrix of $C$ in the basis $\fA$, assuming that columns and rows are sorted in a compatible way with the poset $P$ and dimensions.
That is, we impose a total order in $\fA$ so that, given two generators $\alpha \in \fA^p_n$ and $\beta \in \fA^q_m$ with either $p<q$ or ($p=q$ and $n<m$), then $\alpha < \beta$.

Given $p \in P$ and $n \in \Z$, we denote by $D[p,p,n]$ the restriction of the matrix $D$ to the columns indexed by $\fA^p_n$ and the rows indexed by $\fA^p_{n-1}$.
On the other hand, we denote by $D[:,p]$ the restriction of $D$ to the columns indexed by $\fA^p$ and the rows indexed by $\fA$.
As mentioned above, $D$ denotes the boundary matrix of $C$ in the basis $\fA$. 
In this case $D[p,p,n]$ is the matrix associated to the twisted differential $t^p_n\colon \C^p_n\rightarrow \C^p_{n-1}$. 
while $D[:,p]$ is the matrix associated to the composition $d\circ \imath^p\colon C^p \rightarrow C$.

We perform transformations on $D$ to obtain a connection matrix $\Delta$. 
In order to describe these operations, we refer to columns and rows being labelled by elements from $\fA$, which are the generators of the bases associated to the initial matrix $D$.
In particular, one needs to follow four steps:
\begin{enumerate}[label=(\Roman*)]
    \item\label[step]{algo:relative-reduction} Using the clearing optimisation, reduce the matrices $D[p,p,n]$; this is done by elementary column operations on $D$. 
    Notice that this step can be executed in parallel over $p \in P$. This leads to separating bases $(\fH_n^p, \fB_n^p, \fK_n^p)$ for $\C_n^p$, all $n \in \Z$ and $p \in P$. 
    
    \item\label[step]{algo:optional-removal} (Optional) Remove from $D$ all rows indexed by $\rho(\fK)$ and all columns indexed by $\rho(\fB)$.
    
    \item\label[step]{algo:global-reduction} 
    For each $w \in \fK^p_n$, add the column from $D$ indexed by $\rho(w)$ to other columns sharing the same pivot $\rho(d(w))$, so that the row labelled by the pivot $\rho(d(w))$ has a single non-zero entry (the one for the column indexed by $\rho(w)$). 
    Here, notice that the column indexed by $\rho(w)$ will only be added to columns indexed by generators in $\fA^q$ with $q>p$ (since, at the start of the procedure, $D$ is the matrix of $d$, a $P$-filtered function).
    This operation can be performed following any order, and applying the column operations on $D$.
    
    \item\label[step]{algo:removal} Remove all remaining rows and columns indexed by $\rho(\fK)$ or $\rho(\fB)$.
\end{enumerate}
 
For the reader who wants more details about this procedure, we provide a pseudocode instantiation next.
In our description, the execution of the optional \Cref{algo:optional-removal} is mixed with \Cref{algo:relative-reduction} and we do not discuss parallelisation. The method is summarised in \Cref{alg:connection-matrix}:
\begin{itemize}
    \item \textbf{[lines 1-9]} \Cref{algo:relative-reduction} and \Cref{algo:optional-removal} are executed within the first for loop.
    For each $p \in P$, we start by computing the separating bases $((\fH^p_n, \fB^p_n, \fK^p_n))_{n=0}^N$ as well as the diagonal matrices $(T_n)_{n=0}^N$ by calling \Cref{alg:separating-bases}. After, we consider a matrix $T=\textrm{Diag}((T_n)_{n=0}^N)$ which is the diagonal matrix resulting from putting the matrices $T_i$ for $0 \leq i \leq N$ along the diagonal blocks. 
    Then, \Cref{algo:optional-removal} allows removing a few rows and columns from $D$ and $T$. 
    We end each for loop by replicating the column operations that lead to the separating bases on the column block $D[:,p]$ by simply computing the matrix product $D[:,p]T$.

    \item \textbf{[lines 10-19]} \Cref{algo:global-reduction} is executed within the second for loop. Given $w \in \fK$ and the pivot $\alpha = \rho(w)$, we denote by $\tau$ the pivot from the column $D[:,\alpha]$. 
    Then, this column is added to other columns $D[:,\beta]$ indexed by $\beta \in \fA_{\dim(\alpha)} \setminus \rho (\fB)$ (as we removed the columns $\rho (\fB)$ in \Cref{algo:optional-removal})  such that their pivot is also $\tau$. 
    Since $D$ is the matrix of a $P$-filtered differential, this implies that $\alpha < \beta$. 
    \item \textbf{[lines 20-21]} Finally, \Cref{algo:removal} is executed.
\end{itemize}

\begin{algorithm}
\caption{Connection matrix computation (3 steps)}\label{alg:connection-matrix}
\KwData{The matrix $D$ of $P$-filtered differential $d\colon C\rightarrow C$ with respect to a fixed basis $\fA$ (with an order compatible with $P$).}
\KwResult{The connection matrix $\Delta$}
\For{$p \in P$ \label{for:steps-I-II}}{
    $((\fH^p_n, \fB^p_n, \fK^p_n))_{n=0}^N, (T_n)_{n=0}^N \gets \texttt{separating\_via\_clearing}((D[p,p,n])_{n=0}^N)$\; 
    $T \gets \textrm{Diag}((T_n)_{n=0}^N)$\;
    $D \gets \texttt{remove\_rows}(\rho(\fK^p), D)$\;
    $D \gets \texttt{remove\_columns}(\rho(\fB^p), D)$\;
    $T \gets \texttt{remove\_rows}(\rho(\fB^p), T)$\;
    $T \gets \texttt{remove\_columns}(\rho(\fB^p), T)$\;
    $D[:,p] \gets D[:,p] T$\;
}
\For{$w \in \fK$}{
    $\alpha \gets \rho(w)$\;
    $\tau \gets \rho(D[:,\alpha])$\;
    \For{$\beta \in \fA_{\dim(\alpha)}\setminus  \rho(\B)$}{
    \If{$\rho(D[:, \beta])=\tau$}{ 
    $a \gets D[\tau, \beta] / D[\tau, \alpha]$ 
    \tcp*{Conditional execution implies $\alpha < \beta$}
    $D[:,\beta] \gets a \cdot D[:,\alpha]$\;
    }
    }
}
$D\gets \texttt{remove\_columns}(\rho(\fK), D)$\;
$D \gets \texttt{remove\_rows}(\rho(\fB), D)$\;
\Return $D$ 
\end{algorithm}

After executing this procedure, the result is a matrix $\Delta$ which is associated to a map $\oC\rightarrow \oC[-1]$ with respect to a fixed basis $\zeta(\fH)$
 Correctness of this procedure will follow by showing that the entries from $\Delta$ correspond to the description from \Cref{cor:entry-connection}.
 However, let us first consider an example.
 \begin{example}
     In this example, we compute a connection matrix for the complex $X$ from \Cref{ex:filtration}. To start, we perform~\Cref{algo:relative-reduction} on the differential matrix $D$ associated to $d\colon C(X)\rightarrow C(X)[-1]$ depicted on the left below, where one needs to reduce the relative boundaries, marked by blocks.
     The result is the matrix on the right; the only operation performed consists in adding the column labelled by $\overline{vu}$ to the column labelled by $\overline{vw}$. 
     As a result, the column labelled by $\overline{vw}$ stores the boundary of a sum $d(\overline{vu} + \overline{vw})$.
     Next, we remove the rows indexed by elements from $\fK$ and the columns indexed by elements from $\fB$ (we shade these):
     \[
     \left(\begin{array}{c|ccccccc}
       &
     u & 
     w &
     \overline{uw} &
     v &
     \overline{vu} &
     \overline{vw} &
     \overline{uvw} 
     \\ \hline
     u & 
     \tikznode{m-0-0}{0} & 0 & 1 & 0 & 1 & 0 & 0 \\
     w & 
     0 & \tikznode{m-1-1}{0} & 1 & 0 & 0 & 1 & 0 \\
     \overline{uw} & 
     0 & 0 & \tikznode{m-2-2}{0} & 0 & 0 & 0 & 1 \\
     v & 
     0 & 0 & 0 & \tikznode{m-3-3}{0} & 1 & 1 & 0 \\
     \overline{vu} & 
     0 & 0 & 0 & 0 & 0 & 0 & 1 \\
     \overline{vw} & 
     0 & 0 & 0 & 0 & 0 & \tikznode{m-5-5}{0} & 1 \\
     \overline{uvw} & 
     0 & 0 & 0 & 0 & 0 & 0 & \tikznode{m-6-6}{0} 
     \end{array}
     \right) 
     \rightarrow
     \left(\begin{array}{c|c>{\columncolor{gray!20}}cc>{\columncolor{gray!20}}cccc}
       &
     u & 
     w & 
    \overline{uw} &
     v &
     \overline{vu} &
     \overline{vw} &
     \overline{uvw} 
     \\ \hline
     u & 
     \tikznode{m2-0-0}{0} & 0 & 1 & 0 & 1 & 1 & 0 \\
     w & 
     0 & \tikznode{m2-1-1}{0} & 1 & 0 & 0 & 1 & 0 \\
     \rowcolor{gray!20} \overline{uw} & 
     0 & 0 & \tikznode{m2-2-2}{0} & 0 & 0 & 0 & 1 \\
     v & 
     0 & 0 & 0 & \tikznode{m2-3-3}{0} & 1 & 0 & 0 \\
     \rowcolor{gray!20} \overline{vu} & 
     0 & 0 & 0 & 0 & 0 & 0 & 1 \\
     \overline{vw} & 
     0 & 0 & 0 & 0 & 0 & \tikznode{m2-5-5}{0} & 1 \\
     \overline{uvw} & 
     0 & 0 & 0 & 0 & 0 & 0 & \tikznode{m2-6-6}{0} 
     \end{array}
     \right) 
     \tikz[overlay, remember picture]{%
        \node[thick, box around=(m-0-0)(m-0-0)];
        \node[thick, box around=(m-1-1)(m-2-2)];
        \node[thick, box around=(m-3-3)(m-5-5)];
        \node[thick, box around=(m-6-6)(m-6-6)];
        }
    \tikz[overlay, remember picture]{%
        \node[thick, box around=(m2-0-0)(m2-0-0)];
        \node[thick, box around=(m2-1-1)(m2-2-2)];
        \node[thick, box around=(m2-3-3)(m2-5-5)];
        \node[thick, box around=(m2-6-6)(m2-6-6)];
        }
     \]
    Additionally, we obtain the following decomposition of relative chains:
    \begin{align*}
        \C^0 = & \langle u\rangle = \cH^0_0, 
        &
        \C^1 = & \langle w \rangle \oplus \langle \overline{uw} \rangle = \B^1_0 \oplus \cK^1_1, \\
        \C^2 = & \langle v \rangle \oplus \langle \overline{vu} \rangle \oplus \langle \overline{vu}+\overline{vw} \rangle = \B^2_0 \oplus \cK^2_1 \oplus \cH^2_1, 
        &
        \C^3 = & \langle \overline{uvw} \rangle = \cH^3_2
    \end{align*}
    On the left in the picture below, we show the result after performing \Cref{algo:global-reduction}, which, in this case, consists in adding the column labelled by $\overline{uw}$ to the column labelled by $\overline{vw}$. 
    In this case, the shaded rows must have only a single non-zero entry.
    On the right, we show the result after performing \Cref{algo:removal}, removing the shaded rows and columns; this leads to a connection matrix:
    \[
     \left(\begin{array}{c|c>{\columncolor{gray!20}}c>{\columncolor{gray!20}}ccccc}
       &
     u & 
     \overline{uw} &
     \overline{vu} &
     \overline{vw} &
     \overline{uvw} 
     \\ \hline
     u & 
     0 & 1 & 1 & 0 & 0 \\
     \rowcolor{gray!20} w & 
     0 & 1 & 0 & 0 & 0 \\
     \rowcolor{gray!20} v & 
     0 & 0 & 1 & 0 & 0 \\
     \overline{vw} & 
     0 & 0 & 0 & 0 & 1 \\
     \overline{uvw} & 
     0 & 0 & 0 & 0 & 0 
     \end{array}
     \right) 
     \rightarrow
     \left(\begin{array}{ccc}
        0 & 0 & 0 \\
        0 & 0 & 1 \\
        0 & 0 & 0
     \end{array}
     \right)= \Delta.
     \]
     Also, since $\fH = \fH_0^0 \sqcup \fH^2_1 \sqcup \fH^3_2$ with $\fH_0^0=\{ u\}$, $\fH^2_1=\{\overline{vu}+\overline{vw}\}$ and $\fH^3_2=\{\overline{uvw}\}$, we have that $\zeta(\fH)=\{ u + B_0^0, \overline{vu}+\overline{vw}, \overline{uvw}\}$ is a basis for $\oC=H^0_0 \oplus H^2_1\oplus H^3_2$. Altogether, recall that, the connection matrix $\Delta$ is the matrix associated to the differential $\od\colon \oC\rightarrow \oC_{-1}$ with respect to this basis $\zeta(\fH)$.
 \end{example}
\begin{proposition}\label{prop:connection-matrix-algo}
    The output
    $\Delta$ of \Cref{alg:connection-matrix} is the connection matrix 
    representing $\od$ with respect to the basis $\zeta(\fH)$, which consists of cycle bases of relative homology groups $H^p_n$.
\end{proposition}
\begin{proof}
    To start, denote by $\widetilde{D}$ the resulting matrix after performing \Cref{algo:relative-reduction} and \Cref{algo:global-reduction}, without having executed \Cref{algo:optional-removal}, from the procedure.
    Given $z \in \fH^p_n$, consider the column from $\widetilde{D}$ indexed by $\rho(z)$, which is a chain $v \in C_{n-1}$. 
    Notice that, by construction, it follows that
    \[
    v = d\bigg(z + \sum_{q < p} \sum_{w \in \fK^q_n} wc_w\bigg)
    \]
    for suitable coefficients $c_w \in k$ for all $w \in \fK^q_n$ and all $q < p$. Indeed, only columns indexed by $\rho(w)$ with $w\in \fK^q_n$ can have been added to the column indexed by $\rho(z)$, and, since $\rho(d(w)) \in \fA^q \subset \C^q$, it must be that $q<p$.

    Next, $v\in C_{n-1}=\bigoplus_{q \in P} \C^q_{n-1}$ can be written as $v=\sum_{q \in P} v^q$ for unique chains $v^p\in \C^p_{n-1}$ for all $p \in P$.
    In addition, by the description of \Cref{algo:global-reduction} from the algorithm, it follows that $\jmath_{\rho(\fB)}(v)=0$.
    Now, since, by \Cref{algo:relative-reduction}, $(\fH^q_{n-1}, \fB^q_{n-1}, \fK^q_{n-1})$ is a separating basis, using \Cref{property:separating-base-first} from \Cref{lem:properties-separating-bases}, it follows that $\jmath_B(v)=0$ for all $q \in P$.
    This implies that $v \in \cH \oplus \cK$ and, by \Cref{lem:canonical-chains}, also $\gamma(v)=0$.
    Another consequence, by \Cref{lem:properties-separating-bases} \Cref{property:separating-base-second} is that $\jmath_H(v)=\theta \jmath_{\rho(\fH)}(v)$ and all coordinates from $v$ which are in $\fA \setminus \rho(\fH)$ can be ignored.
    Hence, \Cref{algo:optional-removal} can be executed without affecting the end result since the column operations performed during \Cref{algo:global-reduction} remain the same.
    Notice that the columns optionally removed in \Cref{algo:optional-removal} are indexed by $\rho(\fB)$, so these are never added to other columns. 
    Thus, removing them has no effect on the final result, after executing \Cref{algo:global-reduction} and \Cref{algo:removal}. 
    On the other hand, the rows optionally removed in \Cref{algo:optional-removal} are $\rho(\fK)$ and do not contain relative pivots (i.e. $\rho(\fB)$), so their removal also has no effect on the final result. 
    These optional removals reduce the unnecessary operations when executing \Cref{algo:global-reduction}.

    Next, let us write $y=\sum_{q < p}\sum_{w \in \fK^q_n} wc_w$.
    We claim that $y=\gamma\cS(z)$. 
    If this claim is true, then we have shown the correctness of the algorithm, as $v$ is a chain representing $d(1 + \gamma \cS)\beta(z)$ and $\zeta(z) \in \zeta(\fH)$. The last \Cref{algo:removal} returns the coordinates of $\theta \jmath_{\rho(\fH)}(v)$ with respect to $\zeta(\fH)$; these are the coordinates of the connection matrix by \Cref{cor:entry-connection}.
    
    Let us now prove the claim. 
    In particular, we will show the equivalent statement that $\jmath^s(y) = \jmath^s \gamma\cS(z)$ for all $s \in P$, where $\jmath^s\colon C \twoheadrightarrow \C^s$ denotes the projection map.
    To start, $\jmath^s(y)=0$ for all $s\geq p$ and in this case the claim trivially holds since $\jmath^s\gamma\cS(z)=0$.
    Let us now consider $s \in P$ such that $s < p$ and there does not exist $r \in P$ such that $s<r<p$. 
    In this case, using that $t^{sq}\neq 0$ implies $s\leq q$, it follows that
    \[
    \jmath^s(v) = \jmath^s d(z + y) = 
    t^{sp}(z) + \sum_{q<p} t^{sq}\jmath^q(y) = 
    t^{sp}(z) + t^s\jmath^s(y)\ .
    \]
    Now, since $\gamma\jmath^s(v) = \jmath^s\gamma(v)=0$, and since $y \in \cK$, we deduce that
    $
    \jmath^s(y)=-\gamma t^s\jmath^s(y)= \gamma t^{sp}(z) = \jmath^s \gamma \delta(z).
    $
    Furthermore, using again our hypotheses on $s$, it follows that $\jmath^s(\gamma \delta)^m(z)=0$ for all $m>1$, and so
    $\jmath^s(y) = \jmath^s \gamma \delta(z) = \jmath^s \gamma\cS(z)$.
    We now proceed inductively.
    
    By induction, let $s < p$ and suppose that, for all $q \in P$ such that $s<q<p$, the equality $\jmath^q(y)=\jmath^q \gamma\cS(z)$ holds. 
    In such case
    \begin{multline}\label{eq:induction-step-vs}
    \jmath^s(v) = \jmath^s d(z + y) =
    t^{sp}(z) + t^s\jmath^s(y) + \sum_{s<q<p}t^{sq}\jmath^q(y) \\
    = t^{sp}(z) + t^s\jmath^s(y) + \sum_{s<q<p}t^{sq}\jmath^q \gamma\cS(z) 
    = t^{sp}(z) + t^s\jmath^s(y) + \jmath^s \delta \gamma\cS(z),
    \end{multline}
    where, for the last equality, we have used that $\jmath^s \delta = \sum_{s<q}t^{sq} \jmath^q$ together with the fact that $\jmath^\ell \gamma\cS(z) =0$ for all $\ell \geq p$.
    On the other hand, in order, using $\gamma \jmath^s(v)=0$ together with \cref{eq:induction-step-vs}, $t^{sp}(z)=\jmath^s \delta(z)$, $\gamma\jmath^s=\jmath^s\gamma$, and $\cS=\delta(1 + \gamma\cS)$, it follows
    \begin{multline*}
    \jmath^s(y) = 
    - \gamma t^s\jmath^s(y) = 
    \gamma t^{sp}(z) + \gamma \jmath^s \delta\gamma\cS(z) = 
    \gamma \jmath^s \delta(z) +  \gamma \jmath^s \delta \gamma \cS(z)  =   
    \jmath^s \gamma \delta(1 + \gamma \cS)(z) = \jmath^s\gamma \cS(z).
    \end{multline*}
    This finishes the inductive proof (since $P$ is well-founded) of the claim $y=\gamma \cS(z)$.
\end{proof}

\begin{remark}
    The method presented in this section is very similar to the procedure to obtain the connection matrix from~\cite{Dey2024ComputingReductions, Dey2025ComputingDecomposition}. 
    In particular, after performing \Cref{algo:relative-reduction}, relative boundaries are the same as \emph{homogeneous columns}.
    However, our method differs from those references, since it first computes relative homology groups via the clearing optimisation in \Cref{algo:relative-reduction}. In contrast, the algorithm in~\cite{Dey2024ComputingReductions, Dey2025ComputingDecomposition} proceeds in a single for-loop ranging over the columns of the matrix $D$.
    Our preprocessing of relative homology can be parallelized and provides further computational advantages. 
    Namely, one can execute \Cref{algo:optional-removal}, which deletes rows and columns from $D$, before executing \Cref{algo:global-reduction}, which adds columns indexed by generators at different poset grades.
\end{remark}

\subsection*{Algorithmic complexity}
The method presented in this section has an algorithmic complexity bounded above by $N^3$, where $N=\#\fA$, obtaining a similar bound to the computation of ordinary persistent homology as done in~\cite{Dey2024ComputingReductions,Dey2025ComputingDecomposition}. 
We give a derivation below. 
Recall that $\fA^p$ denote the bases for $\C^p$ for all $p \in P$, which compose a basis $\fA$ for $C$.
Let $M = \max\{\#\fA^p\}$. 
\Cref{algo:relative-reduction} takes a complexity of $NM^2$ times the size of the poset $\#P$; computing separating bases has a complexity of about $M^3$ time and replicating the same reductions on the block $D[:,p]$ is then bounded by $NM^2$. 
On the other hand, let $K=\#\fB=\#\fK$ be the total number of columns being relative boundaries (or preboundaries) only.
Assuming that the optional \Cref{algo:optional-removal} has been executed, \Cref{algo:global-reduction} takes about $K(N-K)^2$ time, since the $K$ relative preboundary columns might be added to the $N-K$ columns that are not boundaries, whose length is also $N-K$. 
Notice that executing \Cref{algo:optional-removal} and \Cref{algo:removal} has a negligible cost,  compared to the previous two.
Thus, the total complexity is of $NM^2(\#P) + K(N-K)^2$. In this case, $K$ is only known after executing \Cref{algo:relative-reduction}, while the other two $M$ and $N$ are known from the start. 
As mentioned earlier, one could parallelize the computation of \Cref{algo:relative-reduction}, which would divide $\#P$ by the number of available processors.

\section{Conclusion and Future Work}
\label{sec:conclusion}

Our definition of the connection matrix is based on the notion of Conley complexes from \cite[Def.~4.23]{harker2021}. We have shown how to combine algebraic Morse theory, in particular algebraic Morse matchings (see~\cref{def:Morse-matching}), with a homological perturbation lemma (\cref{lemma:perturbation}) to obtain the chain contraction of the Conley complex (\cref{thm:main}). 
The upshot of this analysis is an explicit description of the Conley complex differential, stated in \Cref{prop:canonical-differential}, and of the coefficients in the connection matrix, stated in \Cref{cor:entry-connection}. 
In \cref{subsec:splittings-clearing}, we then employed the clearing optimisations from \cite{Chen2011,Bauer2014ClearChunks} to find a splitting of the chain group $C_n$ into a sum of Morse chains, boundaries and pre-boundaries. Conveniently, this procedure leads to a separating basis for the chain group $C_n$ from a Gaussian reduction of the matrix describing the chain complex differential. 
By adding suitable columns labeled by pre-boundaries to the columns of $D$ we then modified $D$ to obtain the connection matrix by following the algorithm outlined in \cref{subsec:algorithm}.

The description of the connection matrix in terms of the infinite sum 
is intriguing: each summand adds homological information from deeper and deeper layers of the poset. A similar approximation of homology can be found in spectral sequences and the spectral systems from~\cite{Matschke22}. 
By regarding a $P$-graded chain complex $(C,d)$ as an $\mathsf{O}(P)$-filtered chain complex (see~\Cref{rmk:lattice}),~\cite{spendlove2025graded} showed that one can obtain a spectral system via its Cartan-Eilenberg system, involving relative homology groups of $\mathsf{O}(P)$-filtered subcomplexes of $(C,d)$. 
Furthermore, by Theorem~5.1. from \cite{spendlove2025graded} such Cartain-Eilenberg systems are invariant under $O(P)$-filtered chain isomorphisms, so a question would be if these systems can be explicitly computed using connection matrices.
As future work, we intend to investigate the computational relationships between the homological perturbation theory presented here and the successive derivations in the spectral systems of the Cartan-Eilenberg system of $(C,d)$. 

Our algorithm relies on homology decompositions, which is a convenience due to only considering chain complexes over vector spaces with field coefficients. Algebraic Morse theory, however, works with coefficients in arbitrary $R$-modules for a commutative ring $R$. Therefore another interesting direction of research would be to adapt the description of the connection matrix and ultimately also the algorithm to this setting. 

Since the connection matrix is dependent on a choice of basis, we can also consider how that choice can be made in light of different applications, especially in dynamical systems, where the Morse indices can be interpreted as homological indices of isolated invariant sets, and the connection matrix as a lower bound on the connecting orbits between them. 

\subsection*{Acknowledgements}

Álvaro Torras-Casas was funded by the French Agence Nationale de la Recherche through the project
reference ANR-22-CPJ1-0047-01. Ka Man Yim's work on this project was supported by a UKRI Future Leaders Fellowship [grant number MR/W01176X/1; PI: J. Harvey]. The authors are grateful to the Welsh government Taith travel grant which enabled this collaboration. 

\printbibliography

@article{harker2021,
    title = {{A computational framework for connection matrix theory}},
    year = {2021},
    journal = {Journal of Applied and Computational Topology},
    author = {Harker, Shaun and Mischaikow, Konstantin and Spendlove, Kelly},
    number = {3},
    month = {9},
    pages = {459--529},
    volume = {5},
    doi = {10.1007/s41468-021-00073-3},
    issn = {2367-1726}
}

@article{Skoldberg2018,
    title = {{Algebraic Morse theory and homological  perturbation theory}},
    year = {2018},
    journal = {Algebra and Discrete Mathematics},
    author = {Sk{\"{o}}ldberg, Emil},
    number = {1},
    pages = {124--129},
    volume = {26}
}

@article{Chen2011,
    title = {{Persistent homology computation with a twist}},
    year = {2011},
    journal = {27th European Workshop on Computational Geometry {\ldots}},
    author = {Chen, Chao and Kerber, Michael},
    number = {3},
    pages = {28--31},
    volume = {45},
    url = {http://citeseerx.ist.psu.edu/viewdoc/download?doi=10.1.1.224.6560&rep=rep1&type=pdf},
    doi = {10.1.1.224.6560}
}

@article{Matschke22,
   author = {Benjamin Matschke},
   doi = {10.1090/tran/8650},
   issn = {1088-6850},
   issue = {9},
   journal = {Transactions of the American Mathematical Society},
   month = {6},
   pages = {6205-6254},
   title = {Successive spectral sequences},
   volume = {375},
   year = {2022}
}

@techreport{Franzosa,
    title = {{The Connection Matrix Theory for Morse Decompositions}},
    year = {1989},
    booktitle = {Source: Transactions of the American Mathematical Society},
    author = {Franzosa, Robert D},
    number = {2},
    month = {2},
    pages = {561--592},
    volume = {311}
}

@article{Franzosa1998AlgebraicTheory,
    title = {{Algebraic transition matrices in the Conley index theory}},
    year = {1998},
    journal = {Transactions of the American Mathematical Society},
    author = {Franzosa, Robert and Mischaikow, Konstantin},
    number = {3},
    pages = {889--912},
    volume = {350},
    doi = {10.1090/S0002-9947-98-01666-3},
    issn = {0002-9947}
}

@book{Weibel1994AnAlgebra,
    title = {{An Introduction to Homological Algebra}},
    year = {1994},
    author = {Weibel, Charles A.},
    month = {4},
    publisher = {Cambridge University Press},
    isbn = {9780521435000},
    doi = {10.1017/CBO9781139644136}
}

@book{Nicolaescu2011AnTheory,
    title = {{An Invitation to Morse Theory}},
    year = {2011},
    author = {Nicolaescu, Liviu},
    publisher = {Springer New York},
    address = {New York, NY},
    isbn = {978-1-4614-1104-8},
    doi = {10.1007/978-1-4614-1105-5}
}

@incollection{Bauer2014ClearChunks,
    title = {{Clear and Compress: Computing Persistent Homology in Chunks}},
    year = {2014},
    author = {Bauer, Ulrich and Kerber, Michael and Reininghaus, Jan},
    pages = {103--117},
    doi = {10.1007/978-3-319-04099-8{\_}7}
}

@article{Dey2024ComputingReductions,
    title = {{Computing Connection Matrices via Persistence-Like Reductions}},
    year = {2024},
    journal = {SIAM Journal on Applied Dynamical Systems},
    author = {Dey, Tamal K. and Lipi{\'{n}}ski, Michał and Mrozek, Marian and Slechta, Ryan},
    number = {1},
    month = {3},
    pages = {81--97},
    volume = {23},
    doi = {10.1137/23M1562469},
    issn = {1536-0040}
}

@article{Dey2025ComputingDecomposition,
   author = {Tamal K. Dey and Andrew Haas and Michał Lipiński},
   doi = {10.1137/25M1739406},
   issn = {1536-0040},
   issue = {1},
   journal = {SIAM Journal on Applied Dynamical Systems},
   month = {3},
   pages = {108-130},
   title = {Computing a Connection Matrix and Persistence Efficiently from a Morse Decomposition},
   volume = {25},
   year = {2026}
}

@article{Joswig2004ComputingFunctions,
    title = {{Computing Optimal Discrete Morse Functions}},
    year = {2004},
    journal = {Electronic Notes in Discrete Mathematics},
    author = {Joswig, Michael and Pfetsch, Marc E.},
    month = {10},
    pages = {191--195},
    volume = {17},
    doi = {10.1016/j.endm.2004.03.038},
    issn = {15710653}
}

@incollection{Mischaikow1995ConleyTheory,
    title = {{Conley index theory}},
    year = {1995},
    booktitle = {Dynamical Systems: Lectures Given at the 2nd Session of the Centro Internazionale Matematico Estivo (C.I.M.E.) held in Montecatini Terme, Italy, June 13–22, 1994},
    author = {Mischaikow, Konstantin},
    editor = {Johnson, Russell},
    pages = {119--207},
    publisher = {Springer Berlin Heidelberg},
    url = {https://doi.org/10.1007/BFb0095240},
    address = {Berlin, Heidelberg},
    isbn = {978-3-540-49415-7},
    doi = {10.1007/BFb0095240}
}

@article{Lipinski2023Conley-Morse-FormanSpaces,
    title = {{Conley-Morse-Forman theory for generalized combinatorial multivector fields on finite topological spaces}},
    year = {2023},
    journal = {Journal of Applied and Computational Topology},
    author = {Lipi{\'{n}}ski, Michał and Kubica, Jacek and Mrozek, Marian and Wanner, Thomas},
    number = {2},
    month = {6},
    pages = {139--184},
    volume = {7},
    doi = {10.1007/s41468-022-00102-9},
    issn = {2367-1726}
}

@book{Mrozek2025ConnectionDynamics,
    title = {{Connection Matrices in Combinatorial Topological Dynamics}},
    year = {2025},
    author = {Mrozek, Marian and Wanner, Thomas},
    publisher = {Springer Nature Switzerland},
    address = {Cham},
    isbn = {978-3-031-87599-1},
    doi = {10.1007/978-3-031-87600-4}
}

@article{Harker2014DiscreteMaps,
    title = {{Discrete Morse Theoretic Algorithms for Computing Homology of Complexes and Maps}},
    year = {2014},
    journal = {Foundations of Computational Mathematics},
    author = {Harker, Shaun and Mischaikow, Konstantin and Mrozek, Marian and Nanda, Vidit},
    number = {1},
    month = {2},
    pages = {151--184},
    volume = {14},
    doi = {10.1007/s10208-013-9145-0},
    issn = {1615-3375}
}

@book{Scoville2019DiscreteTheory,
    title = {{Discrete Morse Theory}},
    year = {2019},
    author = {Scoville, Nicholas A.},
    edition = {},
    volume = {90},
    publisher = {American Mathematical Society},
    address = {Providence, Rhode Island}
}

@article{Curry2016DiscreteCohomology,
    title = {{Discrete Morse Theory for Computing Cellular Sheaf Cohomology}},
    year = {2016},
    journal = {Foundations of Computational Mathematics},
    author = {Curry, Justin and Ghrist, Robert and Nanda, Vidit},
    number = {4},
    month = {8},
    pages = {875--897},
    volume = {16},
    doi = {10.1007/s10208-015-9266-8},
    issn = {1615-3375}
}

@article{Kozlov2005DiscreteComplexes,
    title = {{Discrete Morse theory for free chain complexes}},
    year = {2005},
    journal = {Comptes Rendus. Math{\'{e}}matique},
    author = {Kozlov, Dmitry N.},
    number = {12},
    month = {6},
    pages = {867--872},
    volume = {340},
    doi = {10.1016/j.crma.2005.04.036},
    issn = {1778-3569}
}

@article{Franzosa1986IndexDecompositions,
    title = {{Index filtrations and the homology index braid for partially ordered Morse decompositions}},
    year = {1986},
    journal = {Transactions of the American Mathematical Society},
    author = {Franzosa, Robert},
    number = {1},
    pages = {193--213},
    volume = {298},
    doi = {10.1090/S0002-9947-1986-0857439-7},
    issn = {0002-9947}
}

@book{Conley1978IsolatedIndex,
    title = {{Isolated Invariant Sets and the Morse Index}},
    year = {1978},
    booktitle = {Book},
    author = {Conley, Charles},
    volume = {38},
    publisher = {American Mathematical Society}
}

@article{Batko2020LinkingDecompositions,
    title = {{Linking Combinatorial and Classical Dynamics: Conley Index and Morse Decompositions}},
    year = {2020},
    journal = {Foundations of Computational Mathematics},
    author = {Batko, Bogdan and Kaczynski, Tomasz and Mrozek, Marian and Wanner, Thomas},
    number = {5},
    month = {10},
    pages = {967--1012},
    volume = {20},
    doi = {10.1007/s10208-020-09444-1},
    issn = {1615-3375}
}

@article{Robbin1992LyapunovFunctor,
    title = {{Lyapunov maps, simplicial complexes and the Stone functor}},
    year = {1992},
    journal = {Ergodic Theory and Dynamical Systems},
    author = {Robbin, Joel W. and Salamon, Dietmar A.},
    number = {1},
    month = {3},
    pages = {153--183},
    volume = {12},
    doi = {10.1017/S0143385700006647},
    issn = {0143-3857}
}

@article{Ebli2024MorseComplexes,
    title = {{Morse theoretic signal compression and reconstruction on chain complexes}},
    year = {2024},
    journal = {Journal of Applied and Computational Topology},
    author = {Ebli, Stefania and Hacker, Celia and Maggs, Kelly},
    number = {8},
    month = {12},
    pages = {2285--2326},
    volume = {8},
    doi = {10.1007/s41468-024-00191-8},
    issn = {2367-1726}
}

@article{Harker2021MorseComputation,
    title = {{Morse Theoretic Templates for High Dimensional Homology Computation}},
    year = {2021},
    author = {Harker, Shaun and Mischaikow, Konstantin and Spendlove, Kelly},
    note = {arXiv:2105.0987}
}

@article{Forman1998MorseComplexes,
    title = {{Morse Theory for Cell Complexes}},
    year = {1998},
    journal = {Advances in Mathematics},
    author = {Forman, Robin},
    number = {1},
    month = {3},
    pages = {90--145},
    volume = {134},
    doi = {10.1006/aima.1997.1650},
    issn = {00018708}
}

@article{Mischaikow2013MorseHomology,
    title = {{Morse Theory for Filtrations and Efficient Computation of Persistent Homology}},
    year = {2013},
    journal = {Discrete {\&} Computational Geometry},
    author = {Mischaikow, Konstantin and Nanda, Vidit},
    number = {2},
    month = {9},
    pages = {330--353},
    volume = {50},
    doi = {10.1007/s00454-013-9529-6},
    issn = {0179-5376}
}

@article{Floer1987MorseDiffeomorphisms,
    title = {{Morse theory for fixed points of symplectic diffeomorphisms}},
    year = {1987},
    journal = {Bull. Amer. Math. Soc. (N.S.)},
    author = {Floer, Andreas},
    number = {2},
    pages = {279--281},
    volume = {16},
    url = {https://doi.org/10.1090/S0273-0979-1987-15517-0},
    doi = {10.1090/S0273-0979-1987-15517-0},
    issn = {0273-0979,1088-9485}
}

@techreport{Skoldberg2006MorseViewpoint,
    title = {{Morse Theory from an Algebraic Viewpoint}},
    year = {2006},
    booktitle = {Source: Transactions of the American Mathematical Society},
    author = {Sk{\"{o}}ldberg, Emil},
    number = {1},
    pages = {115--129},
    volume = {358},
    url = {https://www.jstor.org/stable/3845449?seq=1&cid=pdf-}
}

@article{Bott1988MorseIndomitable,
    title = {{Morse theory indomitable}},
    year = {1988},
    journal = {Publications math{\'{e}}matiques de l'IH{\'{E}}S},
    author = {Bott, Raoul},
    number = {1},
    month = {1},
    pages = {99--114},
    volume = {68},
    doi = {10.1007/BF02698544},
    issn = {0073-8301}
}

@book{Knudson2015MorseDiscrete,
    title = {{Morse theory: Smooth and discrete}},
    year = {2015},
    author = {Knudson, Kevin P},
    publisher = {World Scientific Publishing Company}
}

@article{Gugenheim1972OnFibration,
    title = {{On the chain-complex of a fibration}},
    year = {1972},
    journal = {Illinois Journal of Mathematics},
    author = {Gugenheim, V. K. A. M.},
    number = {3},
    month = {9},
    volume = {16},
    doi = {10.1215/ijm/1256065766},
    issn = {0019-2082}
}

@article{Livernet2020OnMulticomplex,
    title = {{On the spectral sequence associated to a multicomplex}},
    year = {2020},
    journal = {Journal of Pure and Applied Algebra},
    author = {Livernet, Muriel and Whitehouse, Sarah and Ziegenhagen, Stephanie},
    number = {2},
    month = {2},
    pages = {528--535},
    volume = {224},
    doi = {10.1016/j.jpaa.2019.05.019},
    issn = {00224049}
}

@article{Gugenheim1974OnProducts,
    title = {{On the theory and applications of differential torsion products}},
    year = {1974},
    journal = {Memoirs of the American Mathematical Society},
    author = {Gugenheim, V. K. A. M. and May, J. Peter},
    number = {142},
    volume = {0},
    doi = {10.1090/memo/0142},
    issn = {0065-9266}
}

@book{Kozlov2021OrganizedTheory,
    title = {{Organized collapse: an introduction to discrete Morse theory}},
    year = {2021},
    author = {Kozlov, Dmitry N},
    volume = {207},
    publisher = {American mathematical society}
}

@article{Dey2022PersistenceSystems,
    title = {{Persistence of Conley--Morse Graphs in Combinatorial Dynamical Systems}},
    year = {2022},
    journal = {SIAM Journal on Applied Dynamical Systems},
    author = {Dey, Tamal K. and Mrozek, Marian and Slechta, Ryan},
    number = {2},
    month = {6},
    pages = {817--839},
    volume = {21},
    doi = {10.1137/21M143162X},
    issn = {1536-0040}
}

@article{Dey2019PersistentDynamics,
    title = {{Persistent homology of Morse decompositions in combinatorial dynamics}},
    year = {2019},
    journal = {SIAM Journal on Applied Dynamical Systems},
    author = {Dey, Tamal K. and Juda, Mateusz and Kapela, Tomasz and Kubica, Jacek and Lipi{\'{n}}ski, Michał and Mrozek, Marian},
    number = {1},
    volume = {18},
    doi = {10.1137/18M1198946},
    issn = {15360040}
}

@article{Wall1961ResolutionsGroups,
    title = {{Resolutions for extensions of groups}},
    year = {1961},
    journal = {Mathematical Proceedings of the Cambridge Philosophical Society},
    author = {Wall, C T C},
    edition = {2008/10/24},
    number = {2},
    pages = {251--255},
    volume = {57},
    publisher = {Cambridge University Press},
    url = {https://www.cambridge.org/core/product/3395B7D6C5CAAD597FEE1B39B6C35CC1},
    doi = {DOI: 10.1017/S0305004100035155},
    issn = {0305-0041}
}

@article{Witten1982SupersymmetryTheory,
    title = {{Supersymmetry and Morse theory}},
    year = {1982},
    journal = {J. Differential Geometry},
    author = {Witten, Edward},
    number = {4},
    pages = {661--692},
    volume = {17},
    url = {http://projecteuclid.org/euclid.jdg/1214437492},
    issn = {0022-040X,1945-743X}
}

@article{Reineck1990TheFlows,
    title = {{The connection matrix in Morse-Smale flows}},
    year = {1990},
    journal = {Transactions of the American Mathematical Society},
    author = {Reineck, James F.},
    number = {2},
    pages = {523--545},
    volume = {322},
    doi = {10.1090/S0002-9947-1990-0972705-3},
    issn = {0002-9947}
}

@inproceedings{Brown1965TheTheorem,
    title = {{The twisted Eilenberg-Zilber Theorem}},
    year = {1965},
    booktitle = {Simposio di Topologia (Messina, 1964)},
    author = {Brown, Ronald},
    editor = {{Edizioni Oderisi}},
    pages = {33--37},
    address = {Gubbio }
}

@article{spendlove2025graded,
  title={Graded Differential Vector Spaces, Cartan-Eilenberg Systems and Conjectures in Conley Index Theory},
  author={Spendlove, Kelly and Vandervorst, Robert},
  journal={Journal of Dynamics and Differential Equations},
  pages={1--32},
  year={2025},
  publisher={Springer},
  doi = {10.1007/s10884-025-10421-x}
}

@BOOK{Roman2008-fh,
  title     = "Lattices and Ordered Sets",
  author    = "Roman, Steven",
  publisher = "Springer",
  month     =  dec,
  year      =  2008,
  address   = "New York, NY",
  doi = {10.1007/978-0-387-78901-9}
}

@ARTICLE{Birkhoff1933-jg,
  title     = "On the combination of subalgebras",
  author    = "Birkhoff, Garrett",
  journal   = "Math. Proc. Camb. Philos. Soc.",
  publisher = "Cambridge University Press (CUP)",
  volume    =  29,
  number    =  4,
  pages     = "441--464",
  month     =  oct,
  year      =  1933,
  doi = {10.1017/S0305004100011464}
}

\end{document}